\documentclass[leqno] {amsart}

\usepackage[active]{srcltx}

\usepackage{pb-diagram}

\usepackage{mathrsfs}
\usepackage{amsmath}
\usepackage{amssymb}
\usepackage{amscd}
\usepackage{amsthm}
\usepackage[latin1]{inputenc}
\usepackage{graphics}
\usepackage{varioref}

\newtheorem{teo}[equation]{Theorem}
\newtheorem{defin}[equation]{Definition}
\newtheorem{remark}[equation]{Remark}
\newtheorem{prop}[equation]{Proposition}
\newtheorem{cor}[equation]{Corollary}
\newtheorem{lemma}[equation]{Lemma}

\labelformat{enumi}{(#1)}

\newcommand{\Keler}             {K\"{a}hler }

\newcommand{\inte}{\operatorname{int}}

\newcommand{\ga}{\gamma}
\newcommand{\Ga}{\Gamma}
\newcommand{\OO}{\mathcal{O}}
\newcommand{\meno}{^{-1}}

\newcommand{\lam}{\lambda}
\newcommand{\PP}{\mathbb{P}}
\newcommand{\Herm}{\mathcal{H}}
\newcommand{\semv}{\mathscr{S}(V)}
\newcommand{\Mom}{\Phi}
\newcommand{\Momw}{\Phi^W}

\newcommand{\Pos}{\mathcal{P}}
\newcommand{\proba}{\mathscr{P}}

\newcommand{\XS}{\overline{X}^S_\tau}
\newcommand{\XF}{\overline{X}^F_M}

\newcommand{\mut}{{\mu_\tau}}
\newcommand{\vut}{{v_\tau}}
\newcommand{\mis}{{\gamma}}

\newcommand{\tmu}{\tilde{\mu}}
\newcommand{\cmis}{\check{\mis}}

\newcommand{\BB}{\mathbb{B}}
\newcommand{\blym}{\bly_\mis}

\renewcommand{\root}{\Delta}
\newcommand{\simple}{\Pi}

\newcommand{\convo}{\widehat{\OO}}

\newcommand{\intec}{\Omega}

\newcommand{\XX}{\mathfrak{X}}
\newcommand{\bly}{{\Psi}}

\newcommand{\liu}{\mathfrak{u}}
\newcommand{\lium}{\mathfrak{u}^-}
\newcommand{\lia}{\mathfrak{a}}

\newcommand{\lieh}{\mathfrak{h}}
\newcommand{\liek}{\mathfrak{k}}
\newcommand{\lier}{\mathfrak{r}}

\newcommand{\KK}{\mathscr{K}}
\newcommand{\liel}{\mathfrak{l}}
\newcommand{\lieg}{\mathfrak{g}}

\newcommand{\lieb}{\mathfrak{b}}
\newcommand{\liep}{\mathfrak{p}}

\newcommand{\liez}{\mathfrak{z}}
\newcommand{\lies}{\mathfrak{s}}
\newcommand{\liem}{\mathfrak{m}}
\newcommand{\liet}{\mathfrak{t}}

\newcommand{\supp}{\operatorname{supp}}
\newcommand{\omfs}{\om_{FS}}
\newcommand{\UIM}{U^{I,-}}
\newcommand{\liumi}{\mathfrak{u}^{I,-}}
\newcommand{\spam}{\,\operatorname{span}\, }

\newcommand{\chern}{\operatorname{c}}

\newcommand{\alfa}{\alpha}
\newcommand{\alf}{\alpha}
\newcommand{\vacuo}{\emptyset}
\newcommand{\otens}{\widehat{\otimes}}
\newcommand{\tsw}{\tilde{S}_W}
\newcommand{\tkw}{\tilde{K}_W}
\newcommand{\ltsw}{\tilde{\lies}_W}
\newcommand{\ttauw}{\tilde{\tau}_W}

\newcommand{\tauc}{$\tau$-connected }
\newcommand{\piw}{\pi_W}
\newcommand{\Piw}{\check{\pi}_{W}}
\newcommand{\pii}{\piw}
\newcommand{\Pii}{{\Piw}}
\newcommand{\ratw} { \PP(V) \setminus \PP(W^\perp)}
\newcommand{\zz}{Z_W}
\newcommand{\zy}{Y_W}
\newcommand{\htau}{H_\mut}
\newcommand{\hh}{h}
\newcommand{\XW}{i_W(X_W)}
\newcommand{\XWn}{i_{W_n}(X_{W_n})}
\newcommand{\XWp}{i_{W'}(X_{W'})}

\newcommand{\ci}{E_W}
\newcommand{\cip}{E_{W'}}

\newcommand{\ccim}{C_W}
\newcommand{\ratp} { \check{p}}
\newcommand{\im}{\operatorname{Im}}
\newcommand{\pf}{_\#}
\newcommand{\FW}{F_W}
\newcommand{\est}{\Lambda}
\newcommand{\cK}{K}
\newcommand{\thw}{\Theta_W}

\newcommand{\la}{\lambda}

\newcommand{\enf}{\emph}
  
\newcommand{\desudt}[1] []      {\dfrac {\mathrm {d} #1 }{\mathrm {dt}}}
\newcommand{\desudtzero}        {\desudt \bigg \vert _{t=0} }

\newcommand{\restr}[1]          {\vert_{#1}}

\newcommand{\vol}{\operatorname{vol}}

\newcommand{\Ad}{\operatorname{Ad}}
\newcommand{\ad}{{\operatorname{ad}}}
\newcommand{\sx}{(}
\newcommand{\xs}{)}
\newcommand{\pai}{\langle}
\newcommand{\ring}{\rangle}
\newcommand{\deo}{\langle}
\newcommand{\ode}{\rangle}
\newcommand{\scalo}{\deo \cdot , \cdot \ode}
\newcommand{\relint}{\operatorname{relint}}

\newcommand{\Lie}{\operatorname{Lie}}
\newcommand{\SU} {\operatorname{SU}}
\newcommand{\su} {\mathfrak{su}}
\newcommand{\Sl}{\operatorname{SL}}
\newcommand{\Gl}{\operatorname{GL}}
\newcommand{\PGl}{\operatorname{\PP GL}}
\newcommand{\gl}{\operatorname{\mathfrak{gl}}}
\newcommand{\Ric}{\operatorname{Ric}}
\newcommand{\demi}{\frac{1}{2}}
\newcommand{\End}{\operatorname{End}}
\newcommand{\lds}{\ldots}
\newcommand{\cds}{\cdots}

\renewcommand{\setminus}{-}
\newcommand{\cinf}{C^\infty}
\newcommand{\tr}{\operatorname{tr}}
\newcommand{\ra}{\rightarrow}
\newcommand{\C}{\mathbb{C}}
\newcommand{\R}{\mathbb{R}}
\newcommand{\debar}{\bar{\partial}}
\newcommand{\KE}{{K\"ahler-Einstein} }
\newcommand{\om}{\omega}

\renewcommand{\phi}{\varphi}
\renewcommand{\bigl}{\left}
\renewcommand{\biggl}{\left}
\renewcommand{\Bigl}{\left}
\renewcommand{\bigr}{\right}
\renewcommand{\biggr}{\right}
\renewcommand{\Bigr}{\right}
\newcommand{\met}{\mathscr{M}}

\newcommand{\rootp}{\root_+(\lieh,\simple)}
\newcommand{\liebp}{\lieb_+(\lieh, \simple)}
\newcommand{\data}{(\lieh, \simple)}
\newcommand{\Data}{\mathscr{D}}

\newcommand{\operp}{\stackrel{\perp}{\oplus}}

\newcommand{\liec}{\mathfrak{c}}
\newcommand{\wdata}{{W,\lieh, \simple}}
\newcommand{\liefw}{\mathfrak{c}_{\wdata}}
\newcommand{\liefwp}{\liefw^-}

\newcommand{\Id}{\mathrm{Id}}

\newcommand{\Sim}{\mathfrak{P}}
\newcommand{\A}{A_W}

\begin{document}

\title{Satake-Furstenberg compactifications, the moment map
  and $\la_1$}

\author{Leonardo Biliotti}

\author{Alessandro Ghigi} 

\begin{abstract}
  Let $G$ be a complex semisimple Lie group, $K$ a maximal compact
  subgroup and $\tau$ an irreducible representation of $K$ on $V$.
  Denote by $M$ the unique closed orbit of $G$ in $\PP(V)$ and by
  $\OO$ its image via the moment map.  For any measure $\mis $ on $M$
  we construct a map $\blym$ from the Satake compactification of $G/K$
  (associated to $V$) to the Lie algebra of $K$. 
  If $\mis $ is the $K$-invariant measure, then $\blym$ is a
  homeomorphism of the Satake compactification onto the convex
  envelope of $\OO$. For a large class of measures the image of
  $\blym$ is the convex envelope. As an application we get sharp upper
  bounds for the first eigenvalue of the Laplacian on functions for an
  arbitrary K\"ahler metric on a Hermitian symmetric space.

  \end{abstract}

  \address{Universit\`{a} di Parma} \email{leonardo.biliotti@unipr.it}
  \address{Universit\`a di Milano Bicocca}
  \email{alessandro.ghigi@unimib.it}


\thanks{ The first author was supported
    by the Project MIUR ``Geometric Properties of Real and Complex
    Manifolds'' and by GNSAGA of INdAM.  The second author was
    supported by the PRIN 2007 MIUR ``Moduli, strutture geometriche e
    loro applicazioni'' and by GNSAGA of INdAM.  } %

  \subjclass[2000]{32M10; 
    53C35; 
    58J50
  }

  \maketitle

\tableofcontents{}

\section{Introduction}

Let $M$ be a closed manifold of dimension $n$ and let $\met(M)$ be the
set of Riemannian metrics on $M$.  For $g\in \met(M)$ denote by
$\lambda_1(M,g)$ the first eigenvalue of $-\Delta_g$, where $\Delta_g$
is 
seminegative the Laplace-Beltrami operator on functions.  The quantity
$\lambda_1(M,g) \vol(M,g)^{2/n}$ is scale invariant.  In 1970 Hersch
\cite{hersch} proved that
\begin{gather*}
  \sup_{g\in \met(S^2)} \lambda_1(S^2,g) \vol(S^2,g) = 8\pi.
\end{gather*}
A similar inequality was proved by Yang and Yau \cite{yang-yau} on
higher genus surfaces.  On the other hand Colbois and Dodziuk
\cite{colbois-dodziuk} proved that if $n\geq 3$, the quantity $
\lambda_1(M,g) \vol(M,g)^{2/n} 
$ is not bounded above (see also \cite{berger-premiere},
\cite{bleecker-spectrum} and \cite{colin-de-verdiere-construction}).
If $M$ is a \Keler manifold and $a\in H^2(M)$, denote by $\KK(a)$ the
set of \Keler metrics with \Keler form in $ a$.  It makes sense to
study
\begin{gather}
  \label{iaap}
  I(a) := \sup_{g\in \KK(a)} \la_1(M, g).
\end{gather}
In 1994 Bourguignon, Li and Yau proved that $I(a)<+\infty$ for a large
class of pairs $(M,a)$ (including all projective ones) and gave an
estimate of $I(a)$ in terms of holomorphic maps $M \ra \PP^n$.  The
main tool in their proof is an interesting geometric construction
relating the symmetric space $\Sl(n+1,\C)/\SU(n+1)$ to the set of
positive definite Hermitian matrices with trace 1.  By the spectral
theorem the latter set is the convex hull of the image of the
``Veronese embedding'', that is the map which sends a line $\ell \in
\PP^n$ to the orthogonal projection on $ \ell$.  If one uses the
obvious isomorphism to substitute Hermitian matrices with trace 1 with
tracefree skew-Hermitian matrices, the ``Veronese embedding'' becomes
just the moment map. This interpretation was already used in a
different direction in \cite{arezzo-ghigi-loi}. Hence the construction
of Bourguignon, Li and Yau relates $\Sl(n+1,\C)/\SU(n+1)$ to the
convex hull of the coadjoint orbit corresponding to $\PP^n$.

The present paper has two purposes: 1) generalize this construction by
replacing $\PP^n$ with an arbitrary flag manifold (i.e. a homogeneous
space $G/P$ with $G$ a complex semisimple Lie group and $P$ a
parabolic subgroup); 2) get sharp upper bounds for $\la_1$ on any
Hermitian symmetric space.

The generalization (1) is as follows.  Let $M=G/P$ be a flag manifold.
The space $\Sl(n+1,\C)/\SU(n+1)$ is replaced by the symmetric space
$X=G/K$, where $K$ is a maximal compact subgroup. Let $\om$ be a
$K$-invariant integral \Keler form, $\Mom: M\ra \liek^*$ the moment
map, $\OO=\Mom(M)$ the corresponding coadjoint orbit and denote by
$\convo$ the convex hull of $\OO$.  $(M,\om)$ is the unique closed
orbit of $G$ in $\PP(V)$ for some irreducible representation $\tau : G
\ra \Gl(V)$.  Using $\tau$ one can construct a Satake compactification
$\XS$ of $X$.  Given a probability measure $\mis$ on $M$ we consider
the \enf{Bourguignon-Li-Yau map}
\begin{gather*}
  \blym : X \ra \liek^* \qquad \blym(gK) = \int_M
  \Mom\bigl(\sqrt{g\theta(g\meno)}\cdot x \bigr) d\mis(x)
\end{gather*}
($\theta$ denotes the Cartan involution; see p. \pageref{radice} for
the square root).  We say that $\mis$ is \enf{$\tau$-admissible} if it
does not charge the hyperplane sections of $M\subset \PP(V)$.  Our
results are the following.
\begin{teo}
  The map $\blym$ extends to $\XS$.
  \label{teoA} If $\mu$ is the $K$-invariant probability measure on
  $M$, then $\bly_\mu$ is a homeomorphism of $\XS$ onto $\convo$.
\end{teo}
Both $\XS$ and $\convo$ have a rich boundary structure, given
respectively by boundary components in the sense of Satake and by open
faces in the sense of convex geometry. The homeomorphism $\bly_\mu$
preserves these structures, sending each boundary component
diffeomorphically onto an open face.
\begin{teo} \label{teoB} If $\mis$ is a $\tau$-admissible measure,
  then $\blym (\XS) = \convo$ and $\blym(\partial \XS)
  \subset \partial \convo$.
\end{teo}
In particular a Satake compactification of a symmetric space $G/K$ of
Cartan type IV (i.e. $G=K^\C)$ is homeomorphic to the convex body
$\Omega:=\inte \convo$ in $\liek$. This fact is not new: Kor\'anyi
\cite{koranyi-remarks} recently showed that for any symmetric space
not necessarily of type IV a Satake compactification $\XS$ is
homeomorphic to a convex body in $\liep$ where $\lieg = \liek \oplus
\liep$.  Already for $\PP^1$ the map used by Kor\'anyi is different
from $\bly_\mu$ and the proofs are completely different as well.  Our
main tool is the moment map and the main point is the link between
$\XS$ and the flag manifold $M$, as described below.

It would be interesting to generalize this construction to Satake
compactifications of symmetric spaces not necessarily of type IV, by
replacing flag manifolds by real flag manifolds. We leave this to
further inquiry.

Following the strategy of Bourguignon, Li and Yau we apply Theorem
\ref{teoB} to get the following result.
\begin{teo}
  \label{intro-3}
  If $M$ is a Hermitian symmetric space of the compact type, then
  $I(2\pi\chern_1(M))$ $ =2$.  The bound is attained by the symmetric
  metric.
\end{teo}
This bound was previously known in the following special cases:
$M=S^2$, proven by Hersch \cite{hersch}, $M=\PP^n$ proven by
Bourguignon, Li and Yau \cite{bourguignon-li-yau}, $M$ the complex
Grassmannian, proven by Arezzo, Loi and second author
\cite{arezzo-ghigi-loi}, $M$ an irreducible symmetric space whose
automorphism group is a classical group, proven in
\cite{biliotti-ghigi-AIF}.  We remark that in the statement above $M$
can be reducible.  It would be interesting to know if this bound holds
more generally for any flag manifold. Our proof breaks down since in
such generality Lemma \ref{lemmazk} is false and the form $\alfa$ is
not even closed.

We now describe the contents of the paper.  Consider a set of data
$G,K, \tau, \scalo$, where $G$ is a connected semisimple complex Lie
group, $K$ is a maximal compact subgroup, $\tau: G \ra \Gl(V)$ is an
irreducible representation with finite kernel and $\scalo$ is a
$K$-invariant Hermitian product on $V$.  Out of these data one can
construct a flag manifold on one side and a Satake compactification
$\XS$ of $X=G/K$ on the other.  This is recalled with some detail in
\S\S \ref{section-flags}-\ref{Satake-section}.  Next (\S
\ref{tauconnessi}) we study the relaton between the compactification
$\XS$ and the flag manifold $M$. The boundary components of $\XS$ have
been described in terms of root data in the pioneering work of Satake
\cite{satake-compactifications}. We recast Satake's analysis in more
geometric terms. In particular $\mut$-connected subsets of the simple
root are replaced by \enf{$\tau$-connected subspaces} $W\subset V$, a
class of subspaces that are particularly well adapted to $M$, e.g.
the intersection $M_W:=M\cap \PP(W)$ is a smaller flag manifold.  By
means of the moment map we show that the submanifolds $M_W$ capture
much of the information contained in $W$. Eventually $\XS$ can be
embedded in the set of rational self-maps of $M$ and the boundary
component corresponding to $M_W$ corresponds to the rational maps $M
\dashrightarrow M_W$ (\S \ref{projection-section}).  \S
\ref{tau-moment-section} is probably the most technical part. Using
the Atiyah-Guillemin-Sternberg convexity theorem and a computation in
terms of root data, we show that each $\tau$-connected subspace $W$
determines a subset $C_W\subset \liek^*$ defined by linear
inequalities, see \eqref{eq:def-CW}, such that $\OO \subset C_W$. If
$E_W$ denotes the set where these inequalities become equalities, then
$\OO\cap E_W = \Mom(M_W)$ (Theorem \ref{sticazzi}). In words, each
$\tau$-connected subspace $W$ gives rise to some inequalities
satisfied by $\convo$ and is responsible for a part of its boundary.
The Bourguignon-Li-Yau map is defined in \S \ref{section-BLY-general}.
The interpretation of $\XS$ in terms of rational maps immediately
yields the extension of $\blym$ to $\XS$, while the results of \S
\ref{tau-moment-section} allow to prove that if $\mis$ is
$\tau$-admissible, then $\blym $ maps the boundary of $\XS$ to the
boundary of $\convo$. In \ref{section-K-inv} we show that for
$\mis=\mu$ (i.e. the $K$-invariant measure) the restriction of
$\bly_\mu$ to $X$ is a local diffeomorphism.  This finally allows to
complete the proofs of Theorems \ref{teoA} and \ref{teoB}.  As a
byproduct of Theorem \ref{teoA} we also get a short proof (only for
type IV) of a theorem by Moore \cite{moore-compactifications}, stating
that Satake and Furstenberg compactifications coincide (\S
\ref{section-furst}).  In \S \ref{sezione-autovalore} we give the
application to $\la_1$ and prove Theorem \ref{intro-3}.

\medskip 
After completing this work we became aware of preprint
\cite{sanyal-sottile-sturmfels-orbitopes} by Sanyal, Sottile and
Sturmfels, which is devoted to the study of \emph{orbitopes}, which
are by definition the convex hulls of orbits of a compact group acting
linearly on a 
 real vector space. Our $\convo$ is an example of
orbitope in the adjoint representation.  Although there is no real
overlap between \cite{sanyal-sottile-sturmfels-orbitopes} and the
present paper and although the points of view are rather different, we
believe that their more general approach should in the future shed
some light on our constructions and conversely the approach in this
paper should be of interest also for more general orbitopes.

\medskip

{\bfseries \noindent{Acknowledgements.}}  We wish to thank Laura
Geatti, Peter Heinzner, Lizhen Ji, Karl-Hermann Neeb, Giorgio
Ottaviani and Luiz San Martin for interesting emails\-/dis\-cussions.

\section{Flag manifolds and compactifications}

\subsection{Flag manifolds}
\label{section-flags}

 \label{data}
 Let $G$ be a connected complex semisimple Lie group and $K\subset G$
 a maximal compact subgroup.  Denote by $\theta :G\ra G $ the Cartan
 involution, such that $K=\operatorname{Fix}(\theta)$.  Let $\lieg$
 and $\liek$ be the Lie algebras and let $\theta$ denote also the
 involution on $\lieg$.  Consider an irreducible complex
 representation $\tau : G\ra \Sl(V)$ and a Hermitian product $\scalo$
 on $V$, which is invariant by $K$. This means that for any $g\in G$
 \begin{gather}
   \label{eq:tau-theta}
   \tau(\theta(g) ) = (\tau(g\meno))^* .
 \end{gather}
 We will always assume that $\ker \tau$ is finite, i.e. that $\tau$ is
 nontrivial on any simple factor of $G$.

 $G, K, \tau$ and $ \scalo$ are the basic data for the construction of
 two different objects: a $G$-homogeneous complex manifold with a
 $K$-invariant Hodge metric on the one side and a compactification of
 the space $X=G/K$ on the other side. The first one is a flag
 manifold, the second is a Satake compactification.  In this \S{} we
 describe the flag manifold.  The definition of the Satake
 compactification will be recalled in \S \ref{Satake-section}.

\begin{remark}
  If we compose $\tau$ with the natural map $\Sl(V) \ra \PGl(V)$ and
  factor by the (finite) kernel of the composition, we get a faithful
  projective representation of some finite quotient of $G$.
  Conversely, any faithtul representation of $G$ can be lifted to a
  linear representation of some finite covering of $G$.  Strictly
  speaking the basic object in the construction of flag manifolds and
  Satake compactifications is the projective representation, see
  \cite[p. 85]{satake-compactifications} and
  \cite[p. 63]{borel-ji-libro}. Nevertheless both constructions are
  not affected by passing to finite quotients/coverings (compare
  \cite[\S 4.1]{guivarch-ji-taylor}).  For simplicity we will deal
  with the representation $\tau : G\ra \Sl(V)$.
\end{remark}

The $K$-invariant Hermitian product on $V$ gives rise to a
$K$-invariant \Keler metric on $\PP(V)$, namely the Fubini-Study
metric associated to $\scalo$. Since $\tau$ is irreducible, in
$\PP(V)$ there is exactly one closed orbit of $G$, which coincides
with the unique complex orbit of $K$ (see
\cite[p. 124]{huckleberry-introduction-DMV}).  We denote this orbit by
$M$ and endow it with the restriction $\om$ of the Fubini-Study
metric. Thus $(M, \om)$ is a \Keler manifold and $K$ acts
transitively, symplectically and almost effectively on $M$.  Let
$\xi_v \in \XX(M)$ denote the fundamental vector field corresponding
to $v\in \liek$.  Since $M$ is simply connected and $K$ is connected
and semisimple the action of $K$ is Hamiltonian with a unique moment
map $ \Mom: M \ra \liek^*$ (see e.g.  \cite[Thm. 26.1 pp.
185-187]{guillemin-sternberg-techniques}), that for any $v\in \liek$
and any $a\in K$ satisfies
\begin{gather}
  \label{eq:mement}
  d \pai \Mom, v\ring = -i_{\xi_v}\om \qquad \Mom(ax)= \Ad^* (a) (
  \Mom(x))=\Phi (x) \circ \Ad (a^{-1} ).
\end{gather}
$\Mom$ is a diffeomorphism of $M$ onto a coadjoint orbit of $K$ that
we denote by $\OO:=\Mom(M)$ (see e.g. \cite[Thm. 32.17 p.
260]{guillemin-sternberg-techniques}).  We call $M$, endowed with all
these structures, the \emph{flag manifold} associated to $G,K$, $\tau$
and $\scalo$.

Denote by $B$ the Killing form of $\lieg$ and by
\begin{gather*}
  \deo \ , \ \ode : \liek^* \times \liek \ra \R
\end{gather*}
the duality pairing.  Fix a Cartan subalgebra $\lieh\subset \lieg$
such that $\theta (\lieh)=\lieh$.  Set
\begin{gather}
  \label{eq:def-liet}
  \liet : = \lieh \cap \liek \qquad \lia:= i \liet.
\end{gather}
$\liet$ is the Lie algebra of a maximal torus $T\subset K$ and $\lieh
= \liet \oplus \lia$.  Denote by $\root(\lieg, \lieh)$ the root system
of $(\lieg, \lieh)$.  Let $\simple$ be a system of simple roots for
$(\lieg, \lieh)$.  Then we say that $\data$ is a \enf{root datum}. (By
this we also understand that $\lieh$ is $\theta$-stable.)  We denote
by $\root_+=\rootp$ the corresponding set of positive roots and put
$\root_-=-\root_+$.  For $\alfa \in \root$ let $H_\alfa \in \lieh$ be
such that $\alfa(X) = B(X,H_\alfa)$ for any $X\in \lieh$.  Also we
transfer $B\restr{\lieh\times \lieh}$ to a bilinear form on $\lieh^*$
by the definition
\begin{gather}
  \label{rappresento}
  B( \alpha, \beta )
  =B(H_{\alpha},H_{\beta})=\alpha(H_{\beta})=\beta(H_{\alpha} )
\end{gather}
for $\alpha, \beta \in \lieh^*$.
Denote by $\liebp$ the standard \enf{positive} Borel subalgebra:
\begin{equation}
  \label{eq:borel}
  \liebp=\lieh \oplus \bigoplus _{\alfa \in \root_+} \lieg_\alfa
\end{equation}
Given a set $I\subset \simple$ of simple roots put
\begin{gather}
  \label{eq:sottoalgebre-1}
  \begin{gathered}
    \root_I = \root \cap \spam (I) \qquad \liep_I = \liebp \oplus
    \bigoplus _{\alfa \in \root_I \cap
      \root_-} \lieg_\alfa\\
    \lieh_I = \bigcap_{\alfa \in I } \ker \alfa \qquad
    \lieh^I = \bigoplus _{\alfa \in I} \C H_\alfa \\
    \lies_I = \lieh^I \oplus \bigoplus _{\alfa \in \root_I}
    \lieg_\alfa \qquad \liu_I = \bigoplus _{\alfa \in \root_+
      \setminus \root_I} \lieg_\alf .
  \end{gathered}
\end{gather}
(E.g. $\liep_\emptyset = \liebp$ and $\liep_\simple = \lieg$.)  Then
$\liep_I$ is a parabolic subalgebra of $\lieg$ and any parabolic
subalgebra containing $\liebp$ is of this form. There is a
decomposition $\lieh = \lieh_I {\oplus} \lieh^I$ which is
$B$-orthogonal. $\liu_I$ is a nilpotent ideal of $\liep_I$, while
$\lies_I$ is a semisimple subalgebra of $\lieg$ which commutes with
$\lieh_I$.  Denote by $P_I, U_I$ and $S_I$ the connected subgroups of
$G$ with Lie algebras $\liep_I, \liu_I$ and $\lies_I$
respectively. They are closed subgroups, $P_I$ is parabolic, $U_I$ is
the unipotent radical of $P_I$, while $S_I$ is semisimple.  Moreover
$\lieh_I\oplus \lies_I$ is a reductive subalgebra and one has the
\emph{Chevalley} (or \emph{algebraic Levi}) \emph{decomposition}
$\liep_I = \liu_I \oplus \lieh_I \oplus \lies_I$ (see
e.g. \cite[p. 32]{ency-lie-algebras-III}).

Once a root datum $\data$ has been fixed, the representation $\tau$
determines the line $x_0=\C\vut$ spanned by any highest weight vector
$\vut$ and $M = G \cdot x_0 = K\cdot x_0$.  The stabilizer of $x_0$ in
$G$ is a parabolic subgroup $P$ which does not contain any simple
factor of $G$, and $K_0=K\cap P$ is the centralizer of a subtorus of
$T$ \cite[Thm. 1]{serre-Borel-Weil}.  The following computation is
well-known (see e.g. \cite{ziegler-Kostant} or
\cite[p.63]{baston-eastwood}, that has the opposite sign convention
for $\Mom$).  We recall the proof for the reader's convenience.
\begin{prop}\label{salute}
  Let $\mut$ be the highest weight of $\tau$ and let $\vut$ be a
  highest weight vector. Set $x_0=[\vut]$ and for any $X\in \lieg$
  write $X=X^\lieh + \sum_{\alfa\in \root}X_\alfa$ where $X^\lieh \in
  \lieh$ and $X_\alfa\in\lieg_\alfa$. Then
  \begin{gather}
    \label{eq:momento-x0}
    \deo \Mom (x_0), X \ode = -i \mut (X^\lieh)
  \end{gather}
\end{prop}
\begin{proof}
  Fix on $\PP(V)$ the Fubini-Study metric induced by $\scalo$.  The
  moment map $\Mom^{\PP(V)}$ of $(\PP(V), \omfs)$ with respect to the
  $\SU(V)$-action is given by the formula
  \begin{gather*}
    \pai \Mom^{\PP(V)}\left( [v] \right) , A \ring = -i \frac{\deo Av,
      v\ode }{|v|^2} \qquad A\in \su(V),\ v\in V.
  \end{gather*}
  Since the inclusion $ M \hookrightarrow \PP(V)$ is $K$-equivariant
  and symplectic $ \Mom^M (x) = \Mom^{\PP(V)} (x) \circ \tau $ where
  $\tau :\liek \ra \su(V)$ denotes the infinitesimal representation.
  So for any $X\in \liek$
  \begin{gather*}
    \pai \Mom^M (x_0),X\ring = -i \frac{\deo \tau(X)\vut, \vut\ode
    }{|\vut|^2}.
  \end{gather*}
  Write $X=X^\lieh + \sum_\alfa X_\alf$ with $X_\alf \in \lieg_\alf$.
  Since $\tau(X_\alfa)\vut \in V_{\mut+\alfa} \perp V_\mut$
  \begin{gather*}
    \deo \tau(X) \vut, \vut\ode = \deo \tau(X^\lieh) \vut, \vut\ode =
    \mut(X^\lieh)|\vut|^2.
  \end{gather*}
  This yields the result.
\end{proof}

\subsection{Satake compactifications}
\label{Satake-section}

Assume given a \enf{real} semisimple Lie group $G$, a maximal compact
subgroup $K\subset G$ and an infinitesimally faithful irreducible
representation $\tau$ of $G$ on a complex vector space $V$, which is
endowed with a $K$-invariant Hermitian product. With these data Satake
\cite{satake-compactifications} constructed a compactification $\XS$
of the symmetric space $X=G/K$.  In this paper we are concerned only
with the case in which $G=K^\C$ (i.e. $X$ is of Cartan type IV).  In
this case the Lie theoretic data simplify considerably since the
restricted roots of the real Lie algebra $\lieg_\R$ underlying $\lieg$
coincide with the roots of the complex Lie algebra $\lieg$.  We wish
to recall the construction of the Satake compactifications and some of
their relevant properties restricting to this particular class of
symmetric spaces and taking advantage of this simplication.  Proofs
can be found in the general case in the book \cite[\S
I.1]{borel-ji-libro} which we follow for most of the notation.

Put
\begin{gather*}
  \Herm (V) = \{ A\in \End(V) : A=A^*\}.
\end{gather*}
Denote by $\pi : \Herm(V)\setminus \{0\} \ra \PP(\Herm(V))$ the
canonical projection and set
\begin{gather*}
  \Pos(V) = \pi(\{A\in \Herm (V) : A >0\}) \subset \PP(\Herm(V)).
\end{gather*}
$\Pos(V)$ consists of points $[A]$ such that $A$ is invertible and all
its eigenvalues have the same sign.
\begin{lemma}
  \label{prop-PV}
  (a) $ \overline { \Pos(V) } =\pi ( \{A\in \Herm (V) : A\neq 0 ,
  A\geq 0\})$.  (b) The restriction of $\pi$ to $ \{A\in \Herm (V) : A
  >0, \det A=1\}$ is a homeomorphism onto $\Pos(V)$.  (c) The
  restriction of $\pi$ to $ \{A\in \Herm (V) : A \geq 0, \tr A=1\}$ is
  a homeomorphism onto $\overline{\Pos(V)}$.
\end{lemma}
\begin{defin}
  For $G,K, \tau, \scalo $ as at p. \pageref{data} set
  \begin{gather}
    \label{eq:imbeddoinpos}
    i_\tau : X: =G/K \ra \Pos(V) \qquad i_\tau(gK) =[ \tau(g)
    \tau(g)^*].
  \end{gather}
  The \enf{Satake compactification} of $X$ associated to $\tau$ and
  $\scalo$ is the space $ \XS:=\overline {i_\tau (X) } $.  The closure
  is taken in $\PP(\Herm(V))$.
\end{defin}
Since $\Sl(V)$ and hence $G$ act on $\PP(\Herm(V))$ by conjugation,
$\XS$ is a $G$-compact\-if\-ication.  We stress that $\XS$ depends only on
$G,K, \tau$ and $\scalo$.

Satake gave a thorough description of the boundary $\partial\XS :=\XS
- i_\tau (X)$ in terms of root data.  We will now recall this
description in our simplified setting. In \S \ref{tauconnessi} we will
reinterpret this description in a way that does not depend on the root
data.

\begin{defin}
Fix a root datum $\data$.
  A subset $E\subset \lia^*$ is \emph{connected} if there is no pair
  of disjoint subsets $D,C\subset E$ such that $D\cup C =E$,
  and $B\sx\lam,\mu\xs=0$ for any $\lam \in D$ and $\mu \in C$.
\end{defin}
(A thorough discussion of connected subsets can be found in \cite[\S
5]{moore-compactifications}.)  Connected components are defined as
usual. For example the connected components of $\simple$ are the
subsets corresponding to the simple roots of the simple ideals in
$\lieg$.  Denote by $\mut$ the highest weight of $\tau$ with respect
to $\data$ and let $v_\tau$ be a highest weight vector.
\begin{defin}
  A subset $I\subset \simple$ is $\mut$-\enf{connected} if
  $I\cup\{\mut\}$ is connected.
\end{defin}
Equivalently, $I$ is $\mut$-connected if and only if every connected
component of $I$ contains at least one element $\alfa$ with $B(\mut,
\alfa) = \mut(H_\alfa)\neq 0$.  By the highest weight theorem, if
$\lam\in \lia^*$ is a weigth of $\tau$, then
\begin{gather*}
  \lam = \mut - \sum_{\alfa \in \simple} c_\alfa \alfa
\end{gather*}
for some nonnegative integers $c_\alfa$. The support of $\lam$ is the
set $ \supp (\lam) : = \{\alfa \in \simple: c_\alfa >0\}$.

\begin{lemma}
  [\protect{\cite[Lemma 5 p. 87]{satake-compactifications}}] $I\subset
  \simple$ is $\mut$-connected if and only if $I=\supp (\lam)$ for
  some weight $\lam$ of $\tau$.
\end{lemma}
For example $\emptyset = \supp(\mut)$ and $\simple$ is
$\mut$-connected since $\tau$ is nontrivial on any simple factor of
$G$.  Given $\lam\in \lieh^*$ denote by $V_\lam$ the corresponding
eigenspace and set
\begin{gather*}
  V_I=\bigoplus_{\supp(\lam) \subset I} V_\lam.
\end{gather*}
\label{def-VI-SI}

\begin{lemma}
  [{\cite[Lemma 8 p. 89]{satake-compactifications}}]
  \label{lemmetto-di-Satake-taui}
  Let $S_I$ be the subgroup of $G$ defined in
  \eqref{eq:sottoalgebre-1}. Then $\tau(g) (V_I) \subset V_I$ for any
  $g\in S_I$ and the representation $\tau_I : S_I \ra \Gl(V_I)$ gotten
  in this way is irreducible.  The highest weight of $\tau_I$ is
  $\mut\restr{\lieh^I}$ and $v_\tau \in V_I$ is a highest weight
  vector.
\end{lemma}

  \begin{defin}
    If $I\subset \simple$ is $\mut$-connected, denote by $I'$ the
    collection of all simple roots orthogonal to $\{\mut\}\cup I$.
    The set $J:=I\cup I'$ is called the $\mut$-\enf{saturation} of
    $I$.
  \end{defin}
  The largest $\mut$-connected subset contained in $J$ is $I$.
  \begin{lemma}
    [{\cite[Prop. I.4.29 p. 70]{borel-ji-libro}}]
    \label{stab-VI}
    If $I$ is $\mut$-connected, then $ P_J=\{g\in G: \tau(g)V_I
    \subset V_I\}$.
  \end{lemma}
  Fix a $\mut$-connected subset $I$. If $A\in \End (V_I)$, let
  $A\oplus 0$ denote the extension of $A$ that is trivial on
  $V_I^\perp$.  If $\pi_I: V\ra V_I$ denotes orthogonal projection and
  $j_I: V_I \hookrightarrow V$ denotes the inclusion, then $A\oplus 0
  =j_I\circ A\circ \pi_I$ and the map
  \begin{gather}
    \label{eq:def-psi}
    \psi : \PP(\Herm (V_I)) \ra \PP(\Herm(V)) \qquad \psi([A]) =
    ([A\oplus 0])
  \end{gather}
  embeds $\Pos(V_I) $ in $\overline{\Pos(V)}$. Note that $K_I=S_I \cap
  K$ is a maximal compact subgroup of $S_I$ and $S_I$ is a semisimple
  complex Lie group. Therefore $X_I=S_I/K_I$ is again a symmetric
  space of type IV.  Hence we have a map $i_{\tau_I} : X_I \ra
  \Pos(V_I)$ defined as in in \eqref{eq:imbeddoinpos}.  Finally define
  \begin{gather*}
    i_I=\psi \circ i_{\tau_I} : X_I \ra \overline{\Pos(V)}.
  \end{gather*}
  \begin{teo}
    [{\cite[Cor. I.4.32]{borel-ji-libro}}]\label{Satakone}
    \begin{gather*}
      \XS = \bigsqcup_{\text{$\mut$-connected $I$}} G .\, i_I(X_I).
    \end{gather*}
  \end{teo}
  If $I=\simple$ then $i_I(X_I) = X$. The sets $g.i_I(X_I)$ with $g\in
  G$ and $I\subsetneq \simple$ are called \emph{boundary components}.

\begin{lemma}
  [{\cite[Prop. I.4.29]{borel-ji-libro}}] \label{Satakotto} The
  boundary components are disjoint: if $I$ is $\mut$-connected and
  $g\in G$ then $g.\, i_I(X_I) \cap i_I(X_I) \neq \emptyset$ if and
  only if $g.\, i_I(X_I) = i_I(X_I) $ if and only if $g\in P_J$.
\end{lemma}

\subsection{$\tau$-connected subspaces}
\label{tauconnessi}

We wish to interpret the construction of Satake compactifications more
intrinsically, i.e. independently of the root data.  Given a subspace
$W\subset V$, $W \neq \{ 0\}$, set
\begin{gather*}
  P_W =\{ g\in G: \tau(g) W \subset W\}\qquad \tkw = K\cap P_W \qquad
  M_W = M\cap \PP(W).
\end{gather*}
The subgroup $P_W$ is closed.
\begin{defin}
  $W$ is a \emph{$\tau$-connected subspace} if $P_W$ is parabolic and
  acts irreducibly on $W$.
\end{defin}

\begin{lemma}
  \label{torello-dentro-KW}
  If $W$ is $\tau$-connected, then $\tkw$ is connected and contains a
  maximal torus of $K$. There is a Cartan subalgebra $\lieh$ of
  $\lieg$ such that $\theta(\lieh) = \lieh$ and $\lieh \subset
  \liep_W$.
\end{lemma}
\begin{proof}
  Since $G/P_W=K/\tkw$ by \cite[Thm. 1]{serre-Borel-Weil} $\tkw$ is
  the centralizer of a torus contained in $K$. Therefore it is
  connected and contains a maximal torus $T$.  If $\liet = \Lie T$,
  then $\lieh := \liet \oplus i \liet$ is a Cartan subalgebra with the
  required properties.
\end{proof}

\begin{cor}
  \label{esistono-data}
  For any \tauc subspace $W$ there are a root datum $\data$ and a
  subset $J\subset \simple$ such that
  \begin{gather}
    \label{eq:W-datum}
    \theta(\lieh) = \lieh \qquad \liebp \subset \liep_W \qquad
    \liep_W=\liep_J.
  \end{gather}
  ($\liep_J$ is as in \eqref{eq:sottoalgebre-1}.)  There are many
  choices for $\data$. If one choice is fixed, then $J$ is unique.
\end{cor}

Denote by $R_W$ the Zariski closure of $\tkw$. It is a complex
connected subgroup of $G$ with Lie algebra $\lier_W =\tilde{ \liek}_W
\otimes \C$.  Let $ \tsw = (R_W,R_W)$ be the commutator subgroup of
$R_W$. It is a connected complex semisimple Lie group.  Therefore it
splits as a product of simple factors $ \tsw = S_1 \times \cds \times
S_r$.  We can reorder the factors in such a way that the action of
$S_1, \lds, S_q$ on $W$ be nontrivial, while the remaing factors
$S_{q+1},\lds, S_r$ act trivially on $W$.  Set
\begin{gather}
  \label{defKW}
  \begin{gathered}
    S_W := S_1 \times \cds \times S_q \qquad
    S'_W :=  S_{q+1} \times \cds \times S_r\\
    K_W=S_W \cap K\qquad K'_W = S'_W \cap K.
  \end{gathered}
\end{gather}
$S_W$ and $S'_W$ are closed complex connected semisimple subgroups of
$G$, while $K_W$ and $K'_W$ are maximal compact subgroups of $S_W$ and
$S'_W$ respectively.  Finally denote by $ \ttauw $ and $\tau_W$ the
representations of $\tsw$ and $S_W$ on $W$ induced by $\tau$.  We
stress that all these definitions depend only on $W$.

\begin{prop}
  \label{W-I}
  Let $W\subset V$ be a $\tau$-connected subspace and let $\data$ and
  $J$ be as in \eqref{eq:W-datum}.  Then
  \begin{enumerate}
  \item \label{W-I-scomp} $\lier_W = \lieh_J \oplus \lies_J$,
    $\liez(\lier_W) =\lieh_J$, $\ltsw=\lies_J$, $\liep_J=\lier_W
    \oplus \liu_J$. In particular $R_W$ is a Levi
    subgroup \label{pagina-fattore-Levi} of $P_W$.
  \item
    \label{W-I-liu_J}
    $\tau$ is trivial on $\liu_J$.
  \item \label{W-I-irreducible} The representations $\ttauw$ and
    $\tau_W$ are irreducible.
  \item \label{W-I-Ji} If $J_1, \lds, J_r$ are the connected
    components of $J$, then $S_{J_i}$ are the simple factors of
    $\tsw$.
  \item \label{W-I-Cartan} $\lieh^{J_i}$ is a Cartan subalgebra of
    $\Lie S_i=\lies_{J_i}$ and $\lieh^J=\oplus_{i=1}^r \lieh^{J_i}$ is
    a Cartan subalgebra of $\ltsw$.
  \end{enumerate}
  Assume by reordering that $S_{J_i}$ acts nontrivially on $W$ if and
  only if $i\leq q$. Set $ I:=J_1 \sqcup \cds \sqcup J_q $, $
  I':=J_{q+1} \sqcup \cds \sqcup J_r$.  Denote by $\mut$ the highest
  weight of $\tau$ with respect to $\data$.  Then
  \begin{enumerate}
    \setcounter{enumi}{5}
  \item
    \label{W-I-ricartan}
    $\lies_W=\lies_I$, $\lies_W'=\lies_{I'}$, $\lieh^J$ is Cartan
    subalgebra of $\lies_J$, $\lieh^I$ is Cartan subalgebra of
    $\lies_W$ and $\lieh^{I'}$ is Cartan subalgebra of $\lies_W'$.
  \item \label{W-I-sistemi} $ \simple':= \{\alfa\restr{\lieh^J} :
    \alfa \in J\} $ is a system of simple roots for $(\lies_J,
    \lieh^J)$ and $ \simple'':= \{\alfa\restr{\lieh^I} : \alfa \in I\}
    $ is a system of simple roots for $(\lies_W, \lieh^I)$.
  \item $\lieh^J = \lieh^I \operp \lieh^{I'}$,
    \label{lieh-decomposition}
    $\lieh = \lieh_J \operp \lieh^I \operp \lieh^{I'}$, $\lieh_I =
    \lieh_J \operp \lieh^{I'}$, $\lieh_{I'} = \lieh_J \operp
    \lieh^{I}$.
  \item If $\vut$ is a highest weight vector of $\tau$ with respect to
    $\data$, then $\vut \in W$ and $\vut$ is also a highest weight
    vector of $\ttauw$ with respect to $(\lieh^J,\simple')$ and of
    $\tau_W$ with respect to $(\lieh^I,\simple'')$.\label {W-I-peso}
  \item If $x_0=[\vut]$ then $\tsw\cdot x_0 =S_W \cdot x_0 \subset
    M_W$; in particular $M_W\neq \vacuo$;
\label {W-I-orbite}
  \item \label{W-I-saturo} $I$ is $\mut$-connected and $J$ is its
    $\mut$-saturation.
  \item \label{W-I-V-I} $W=V_I$.
  \end{enumerate}
\end{prop}
\begin{proof}
  \ref{W-I-scomp} The first statement follows by writing elements of
  $\lieg$ in a basis adapted to the compact form as e.g. in
  \cite[pp. 352ff]{knapp-beyond}).  The rest follows immediately.
  \ref{W-I-liu_J} Since the representation of $P_W$ on $W$ is
  irreducible and $[\liep_W , \liu_J ] \subset \liu_J$, it follows
  from Engel theorem that $U_J=\exp \liu_J$ acts trivially and that
  the representation of $R_W$ on $W$ is irreducible.
  \ref{W-I-irreducible} By Schur lemma the elements of $\exp\lieh_J$
  act as scalars and the representation $\ttauw$ is irreducible.
  Since $\tau_W(S_W) = \ttauw (\tsw)$ by the definition of $S_W$, the
  representation $\tau_W$ is irreducible as well.  \ref{W-I-Ji}-\ref
  {W-I-sistemi} follow immediately from the definitions of $ \lies_J$,
  $\lies_W$ and $\lies_W'$, see \eqref{eq:sottoalgebre-1}.
  \ref{lieh-decomposition} By construction $J=I\sqcup I'$ and $I\perp
  I'$, so $\lieh^J = \lieh^I \operp \lieh^{I'}$.  Since $\lieh=\lieh_J
  \operp \lieh^J$, the second statement follows.  Next observe that
  $\lieh_J \subset \lieh_I$, since $I\subset J$, and $I'\perp I$, so
  $\lieh^{I'} \subset \lieh_I$. Since $\lieh_J \perp \lieh^{I'}$,
  $\dim \lieh_J \oplus \lieh^{I'} = \dim \lieh - |J | + |I'| = \dim
  \lieh_I$.  This proves the third statement and the same argument
  yields the fourth.  \ref{W-I-peso} In passing from $\tau$ to
  $\ttauw$ we restrict both the group and the space.  Therefore some
  care is needed since a priori we don't know that $\vut$ belongs to
  the smaller space, i.e. $W$.  So we start by fixing a highest weight
  vector $w \in W$ of $\ttauw$ with respect to $ \lieh^J$.  Since any
  element of $\lieh_J$ acts on $W$ by scalar multiplication, $w$ is an
  eigenvector of $\lieh = \lieh_J \oplus \lieh^J$.  To check that $w$
  is a highest weigth of $\tau$ with respect to $\lieh$ and $\simple$
  we need to show that $\tau(\lieg_\alf)w=0$ for any $\alf\in
  \root_+$.  If $\alfa\in \root_+$, then $\lieg_\alf \subset \liep_W$
  so either $\lieg_\alf \subset \liu_J$ or $\lieg_\alf \subset
  \lies_W=\lies_J$. In the first case $\tau(\lieg_\alf)w=0$ since
  $\tau$ is trivial on $\liu_J$ by \ref{W-I-liu_J}.  In the second
  case $\tau(\lieg_\alf)w=0$ since $w$ is the highest weight vector of
  $\ttauw$ and $\alfa \in \root_{J,+}$.  This shows that
  $\tau(\lieg_\alf)w=0$ for any $\alf\in \root_+$ and since $\tau$ is
  irreducible we get that $w$ is also a highest weight of $\tau$, so
  we denote it by $\vut$.  In passing from $\ttauw$ to $\tau_W$ we
  only restrict the group (not the representation space!) so it is
  immediate that $\vut$ is also a highest weight vector of $\tau_W$
  with respect to $\lieh^I=\lieh^J\cap \lies_I$ and $\simple''$.  \ref
  {W-I-orbite} By the definition of $S_W$ we have $\tsw\cdot x_0 = S_W
  \cdot x_0 \subset \PP(W)$ and of course $\tsw \cdot x_0 \subset
  G\cdot x_0$. By Borel-Weil theorem $G\cdot x_0=M$ since $\vut$ is a
  highest weight vector.  \ref{W-I-saturo} We need to show that $\mut$
  is orthogonal to $J_i$ iff $i>q$.  By
  \cite[p.197]{goodman-wallach-Springer} there are irreducible
  representations $\sigma_ i : S_{i} \ra \Gl(V_i)$ for $i=1, \lds, r$,
  such that $W=V_1 \otimes \cds \otimes V_r$ and $\ttauw =
  \sigma_1\otens \cds \otens \sigma_r$.  The factor $S_i$ acts
  trivially on $W$ iff $\sigma_i$ is the trivial representation, which
  is equivalent to $\mut\restr{\lieh^{J_i}}=0$ i.e. to $\mut$ being
  orthogonal to $J_i$.  So indeed $\mut$ is orthogonal to $J_i$ iff
  $i>q$.  \ref{W-I-V-I} We just proved that $W=V_1\otimes \cds \otimes
  V_q$, $\tau_W = \sigma_1 \otens \cds \otens \sigma_q$ and $S_I=S_W$.
  So both $W$ and $V_I$ are irreducible $S_I$-submodules of $V$.
  Since $\vut$ belongs to both, they must coincide.
\end{proof}
In Theorem \ref{orbitina} we will show that equality holds in
\ref{W-I-orbite}.

\begin{cor}
  \label{MW-determina}
  If $W \subset V$ is $\tau$-connected, the linear span of $M_W$ is
  $\PP(W)$. If $W_1, W_2 \subset V$ are $\tau$-connected subspaces,
  then $W_1 \subset W_2$ if and only if $M_{W_1} \subset M_{W_2}$.
\end{cor}
\begin{proof}
  $M_W$ is a nondegenerate subvariety of $\PP(W)$ since it contains
  the closed orbit $S_W\cdot x_0$ of the irreducible representation
  $\tau_W$. The result follows.
\end{proof}

\begin{prop}
  Let $\data$ be a root datum and let $\mut$ be the highest weight of
  $\tau$. Fix a $\mut$-connected subset $ I\subset \simple$ and $g\in
  G$.  Then
  \begin{enumerate}
  \item [(a)] $W=\tau(g) V_I$ is a $\tau$-connected subspace.
  \item[(b)] There is some $a\in K$ such that $W=\tau(a)V_I$.
  \item[(c)] $ P_W= g P_J g\meno$, $S_W = g S_I g \meno$ and $\tsw = g
    S_J g\meno$, where $J$ is the $\mut$-saturation of $I$.
  \end{enumerate}
\end{prop}
\begin{proof}
  It follows from Lemmata \ref {lemmetto-di-Satake-taui} and
  \ref{stab-VI} and from the analysis in the proof of Proposition
  \ref{W-I} that $V_I$ is a $\tau$-connected subspace and that
  $S_{V_I}=S_I$, $P_{V_I}=P_J$. If $W=\tau(g)V_I$ then
  $P_W=gP_Jg\meno$ so in particular $P_W$ is parabolic.  Since
  $G=KP_J$, $g=ap$ for some $a\in K$, $p\in P_J$. Therefore $W=\tau(a)
  V_I$ and $ P_W = a P_J a\meno$ acts irreducibly on $W$, so $W$ is
  $\tau$-connected.  This proves (a) and (b). (c) follows immediately.
\end{proof}

\begin{cor}
  Let $\data$ be a root datum and let $\mut$ be the highest weight of
  $\tau$. A subspace $W\subset V$ is $\tau$-connected if and only if
  $W=\tau(a) V_I$ for some $a\in K$ and for some $\mut$-connected
  subset $I\subset \simple$. In this case $P_W=a P_J a\meno$ and $S_W=
  a S_I a\meno$, where $J$ denotes the $\mut$-saturation of $I$.
\end{cor}

We can now reformulate Satake's analysis of the boundary of $\XS$.  If
$W \subset V$ is a $\tau$-connected subspace, the data $S_W, K\cap
S_W, \tau_W , \scalo $ are again of the type described at
p. \pageref{data}.  So we can set $ X_W : = S_W / K\cap S_W$ and there
is an embedding analogous to \eqref{eq:imbeddoinpos}
\begin{gather}
  i_{\tau_W} : X_W \ra \Pos(W) \qquad i_{\tau_W}(gK_W) =[ \tau_W(g)
  \tau_W(g)^*].
\end{gather}
Define
\begin{gather*}
  i_W:=\psi \circ i_{\tau_W} : X_W \ra \overline{\Pos(V)}
\end{gather*}
where $\psi$ is as in \eqref{eq:def-psi}.  Theorem \ref{Satakone} and
Lemma \ref{Satakotto} can be rephrased in the following way.

\begin{teo}\label{ziasti}
  The boundary components of $\XS$ are exatcly the subsets of $\XS$ of
  the form $\XW$ for some \tauc subspace $W \subsetneq V$, while
  $X=i_V(X_V)$. Moreover
  \begin{gather}
    \label{eq:WSatake}
    \XS = \bigsqcup_{W \text{ $\tau$-connected}} \, \XW.
  \end{gather}
\end{teo}

\subsection{Projections and rational maps}
\label{projection-section} 
In this \S{} we will study the projection $\PP(V) \dashrightarrow
\PP(W)$ induced by the decomposition $V=W\oplus W^\perp$, where $W$ is
a \tauc subspace. Next we will interpret elements of $\XS$ as rational
self-maps of $M$.

For any $\tau$-connected subspace $W$, denote by $ \piw : V \ra W$ the
orthogonal projection and by $\Piw$ its projectivization:
\begin{gather}
  \label{eq:def-ratw}
  \Piw : \ratw \ra \PP(W) \qquad \Piw ([v]) = [\piw(v)].
\end{gather}
Both $\piw$ and $\Piw$ are $S_W$-equivariant, since the splitting $V=W
\oplus W^\perp$ is preserved by $K_W$.

Let $\data$, $\mut$, $\vut$, $x_0$, $I$ and $J$ be as in Proposition
\ref{W-I}.  Let $P$ be the stabilizer of $x_0$ and let $E\subset
\simple$ be such that $P=P_E$.  Set
\begin{gather*}
  \lium:= \bigoplus_{\alfa \in \root_- \setminus \root_E } \lieg_\alfa
  \qquad \root _{I,-} := \root_I \cap \root_- \qquad \liumi :=
  \bigoplus_{\alfa \in \root_{I,-} \setminus \root_E } \lieg_\alfa.
\end{gather*}
By construction $\liumi \subset \lium$.  Let $ \mu : \lium \ra \liumi
$ be the projection according to the root space decomposition: if $X=
\sum
X_\beta\in \lium$ with $X_\beta \in \lieg_\beta$, then
\begin{gather*}
  \mu ( X): = \sum_{\beta \in \root_{I,-} - \root_E} X_\beta.
\end{gather*}
\begin{lemma} \label{piitau} If $X\in \lium$ and $ k\in \mathbb{N}$,
  then $ \pii(\tau(X)^k \cdot v_\tau ) = \tau(\mu(X))^k\cdot v_\tau$.
\end{lemma}
\begin{proof}
  Since $\vut\in V_I=W$, $\pii(v_\tau)=\vut$ so the statement is true
  for $k=0$.  If $k>0 $ write $ X = \sum_{\beta \in \root_- - \root_E}
  X_\beta $ with $ X_\beta \in \lieg_{\beta} $.  Then
  \begin{gather*}
    \tau(X)^k \cdot v_\tau = \sum_{\beta_1, \lds , \beta_k\in \root_-
      - \root_E} \tau(X_{\beta_1} ) \cds \tau(X_{\beta_k}) \cdot
    v_\tau
  \end{gather*}
  Since $ \tau(X_{\beta_1} ) \cds \tau(X_{\beta_k})\cdot v_\tau \in
  V_{\mut + \beta_1 + \cds + \beta_k} $,
  \begin{gather*}
    \pii \Bigl (\tau(X_{\beta_1} ) \cds \tau(X_{\beta_k}) \cdot v_\tau
    \Bigr ) =
    \begin{cases}
      \tau(X_{\beta_1} ) \cds \tau(X_{\beta_k})\cdot v_\tau & \text{if
      }
      \supp( \beta_1 +\cds + \beta_k) \subset I\\
      0 & \text{otherwise}.
    \end{cases}
  \end{gather*}
  All the $\beta_i$'s are negative roots, so if one of them has
  nonzero component in the direction of some $\alfa\in I$, this
  component survives in the sum.  Hence $ \supp( \beta_1 +\cds +
  \beta_k) \subset I $ iff $\beta_i \in \root_I$ for $i=1,\lds, k$.
  Therefore
  \begin{gather*}
    \pii(\tau(X)^k \cdot v_\tau ) = \sum_{\beta_1, \lds , \beta_k\in
      \root_{I,-} - \root_E} \tau(X_{\beta_1} ) \cds
    \tau(X_{\beta_k})\cdot v_\tau .
  \end{gather*}
  On the other hand
  \begin{gather*}
    \mu(X) = \sum_{\beta \in \root_{I,-} - \root_E} X_\beta\\
    \tau(\mu(X))^k \cdot \vut = \sum_{\beta_1, \lds , \beta_k\in
      \root_{I,-} - \root_E} \tau(X_{\beta_1} ) \cds
    \tau(X_{\beta_k})\cdot v_\tau = \pii(\tau(X)^k \cdot v_\tau ).
  \end{gather*}
  This proves the lemma.
\end{proof}
Set
\begin{gather*}
  U ^- := \exp \lium \qquad \UIM := \exp \liumi.
\end{gather*}
\begin{lemma}
  \label{dominante}
  $\lium$ and $\liumi$ are nilpotent subalgebras of $\lieg$.  $U ^-$
  and $\UIM$ are closed connected algebraic subgroups of $G$ and $S_W=S_I$
  respectively.  $U ^-\cap P=\{e\}$.  The maps $ \exp : \lium \ra U ^-
  $ and $ \exp : \liumi \ra\UIM $ are diffeomorphisms.  The map $ a
  \mapsto \tau(a) \cdot x_0 $ is a dominant map of $ U^- $ to $ M $
  and its restriction to $\UIM$ is a dominant map of $\UIM$ to
  $S_W\cdot x_0$.
\end{lemma}
For a proof see e.g. \cite[p. 51]{ottaviani-rat} or
\cite[p. 68]{akhiezer-libro}.

\begin{teo}
  \label{orbitina}
  For any \tauc subspace $W$, $ \Pii \left(M \setminus \PP(W^\perp)
  \right ) = S_W \cdot x_0 = M_W$.
\end{teo}
\begin{proof}
  If $X\in \lium$ then
  \begin{gather*}
    \tau(\exp X) = \sum_{k=0}^\infty \frac{\tau(X)^k}{k!}  \qquad
    \tau(\exp\mu( X)) = \sum_{k=0}^\infty \frac{\tau(\mu (X))^k}{k!}
  \end{gather*}
  and the sums are finite since $\tau(X)$ and $\tau(\mu(X))$ are
  nilpotent. Using Lemma \ref{piitau} we get
  \begin{gather*}
    \pii(\tau(\exp X)\cdot \vut)= \sum_{k=0}^\infty \frac{1}{k!}
    \pii(\tau(X)^k \cdot \vut) = \sum_{k=0}^\infty \frac{1}{k!}
    \tau(\mu(X))^k \cdot \vut = \\= \exp\tau(\mu(X)) \cdot \vut =
    \tau(\exp\mu(X)) \cdot \vut.
  \end{gather*}
  This proves that
  \begin{gather}
    \label{eq:prima-prova}
    \pii(U^-\cdot\vut ) \subset \UIM\cdot\vut .
  \end{gather}
  Since $\vut \neq 0$, $0\not \in \pii(U^-\cdot \vut) $, i.e. $
  U^-\cdot \vut \cap W^\perp = \vacuo$. Therefore $U^-\cdot x_0
  \subset\ratw$.  By Lemma \ref{dominante} $U^-\cdot x_0$ is dense in
  $G\cdot x_0$. Hence it follows from \eqref{eq:prima-prova} that
  \begin{gather}
    \label{eq:mappettonaW}
    \Pii \left(M \setminus \PP(W^\perp) \right) \subset S_W\cdot x_0.
  \end{gather}
  On the other hand $S_W$ preserves $\PP(W)$, so $S_W\cdot \vut \cap
  W^\perp = \vacuo$ and
  \begin{gather*}
    S_W\cdot x_0 = \Pii( S_W \cdot x_0) \subset \Pii\left (M \setminus
      \PP(W^\perp) \right ).
  \end{gather*}
  This proves that $ \Pii \left(M \setminus \PP(W^\perp) \right ) =
  S_W \cdot x_0 $. If $g\in S_W$, then $g\cdot x_0 \in G\cdot x_0 =M$
  and $g\cdot x_0 \in \PP(W)$ since $S_W$ preserves $W=V_I$. Hence $
  S_W \cdot x_0 \subset M\cap \PP(W)=M_W $. Conversely, if $x\in M_W$,
  then $x\not \in \PP(W^\perp)$ and $\Pii(x) =x $. But $x\in M$, so
  $\Pii(x) \in S_W \cdot x_0$ by \eqref{eq:mappettonaW}. Therefore $
  M_W \subset S_W \cdot x_0$.  This proves that $S_W\cdot x_0 = M_W$.
\end{proof}

We can now give the interpretation in terms of rational maps.  For
technical reasons that will become clear later (Theorem
\ref{infiammata}) we prefer to take a square root. Recall the
following elementary fact.
\begin{lemma}
\label{sqrt-cont}
Let $V$ be a Hermitian vector space and set $ \semv =\{A\in \Herm(V):
A\geq 0\} $. If $A\in \semv$ there is a unique
$B\in \semv $ such that $B^2=A$. Set $\sqrt{A}:=B$. Then
$\sqrt{\cdot}$ is a homeomorphis of $\semv$ onto itself.
\end{lemma}
\begin{proof}
  It is enough to prove that $q: \semv \ra \semv $, $q(A)=A^2$ is a
  homeomorphism.  $q$ is of course continuous and using the spectral
  theorem one easily proves that it is bijective.  It is enough to
  show that $q$ is proper.  Let $\{A_n\}$ be a sequence in $ \semv$
  such that $q(A_n) =A^2_n \ra B$.  Then $ \tr A_n ^* A_n = \tr A^2_n \to \tr
  B$.  Therefore $\tr A^*_n A_n$ is bounded and $A_n$ admits a
  subsequence that converges to some $A$, which necessarily lies in
  $\semv$ since $\semv$ is closed in $\Herm(V)$.
\end{proof}
\label{radice}
Let $G, K, \tau$ and $ \scalo$ be as at p. \pageref{data}.  Then
\begin{gather*}
  \Sim:=\{g\in G : \theta(g) = g\meno\}
\end{gather*}
is a submanifold of $G$, $\exp : i\liek \ra \Sim$ is a diffeomorphism
with inverse $\log : \Sim \ra i\liek$,
the map
\begin{gather*}
  g \ra \sqrt{g}:=\exp \left( \frac{\log g}{2} \right )
\end{gather*}
is a diffeomorphism of $\Sim$ onto itself and the map $\Sim\ra X$, $g\mapsto
gK$ is a diffeomorphism. Set
\begin{gather}
\label{eq:df-rho}
  \rho : G \ra \Sim \qquad \rho(g) := \sqrt{g\theta(g\meno)}.
\end{gather}
Then $a = \rho(g)\meno g \in K$ and $g= \rho(g) \cdot a$ is the polar
decomposition of $g$.
\begin{defin}
  If $p=[A]\in \XS $, let $R_p : \PP(V) \dashrightarrow \PP(V)$ be the
  rational map induced by the endomorphism $\sqrt{A}\in \semv$,
  i.e. $R_p$ is defined outside of $\PP(\ker A)$ and for $x=[v]\not
  \in \PP(\ker A)$, $R_p(x) = [\sqrt{A} v]$.  Set
\begin{gather}
\label{eq:def-ratp}
\ratp:= R_p \restr{M} : M \dashrightarrow M.
\end{gather}
\end{defin}
For this to make sense we need to show that $M$ intersects the domain
of definition of $R_p$ and that $R_p(M) \subset M$. 
\begin{lemma}
   \label{rat-map}
   Fix $g\in S_W$ and set $p=i_W(gK_W) \in \XS$.  The following hold.
  \begin{enumerate}
  \item \label{sqrat-indeterminacy} The indeterminacy locus of $R_p$
    is $\PP(W^\perp)$.
  \item \label{sqrat-2} For $x=[v]\in \PP(V) \setminus \PP(W^\perp)$
    \begin{gather*}
      R_p(x) =\left [ \sqrt{\tau(g)\tau(g)^*} \piw(v)\right].
    \end{gather*}
  \item \label{sqrat-3} $M_W \subset M\setminus \PP(W^\perp)$. 
  \item \label{sqrat-3b} If $x\in M_W$, then $R_p(x) = \rho(g) \cdot
    x$. 
  \item \label{ratiosu} 
 $R_p\left (M\setminus \PP(W^\perp)\right) = M_W$.
\item \label{sqrat-5} $\ratp$ is a map as in \eqref{eq:def-ratp} and
  $\ratp = L_{ \rho(g)} \circ \Piw$, where $L_{\rho(g)}$ denotes the
  automorphism of $M_W$ induced by $\rho(g)$.
\item \label{sqrat-6}
If $p\in \XS$, then $p\in i_W(X_W) $ iff $\im \ratp=M_W$.
  \end{enumerate}
  \begin{proof}
    \ref{sqrat-indeterminacy} Set $A=\tau_W(g) \tau_W(g)^* \oplus
    0$. Then $p=[A]$ and $R_p$ is the rational map defined by the
    endomorphism $\sqrt{A} \in \End V$.  Therefore the indeterminacy
    locus is $\PP(\ker A) = \PP(W^\perp)$.  \ref{sqrat-2} and
    \ref{sqrat-3} are obvious.  \ref{sqrat-3b} Set $g_1=\rho(g), a=
    g_1\meno g$. Then $g\theta(g\meno) = g_1^2$, $
    A=\tau_W(g)\tau_W(g)^* = \tau_W(g_1)^2$, so
    $\sqrt{A}=\tau_W(g_1)$. If $x=[v] \in M_W$, then $R_p(x) =
    [\sqrt{A} \, \piw (v) ] = [\tau_W(g_1) v] = g_1 \cdot x$.
    \ref{ratiosu} By what we just proved and Theorem \ref{orbitina}
    $R_p\left (M \setminus \PP(W^\perp) \right ) = L_{g_1} \Piw \left
      (M\setminus \PP(W^\perp)\right) = L_{g_1}(M_W)=M_W$.
    \ref{sqrat-5} By \ref{sqrat-3} $M\not \subset \PP(W^\perp)$ so
    $\ratp=R_p\restr{M}$ is a well-defined rational map and by
    \ref{ratiosu} its image is contained in $M$. The rest is just a
    restatement of \ref{ratiosu}.  \ref{sqrat-6} This follows from
    \ref{ratiosu} and Corollary \ref{MW-determina}.
  \end{proof}
\end{lemma}

\subsection{\tauc subspaces and the moment map}
\label{tau-moment-section}

In this \S{} we study \tauc subspaces from the point of view of the
coadjoint orbit $\OO = \Mom(M)$. We will show that to each \tauc
subspace $W$ there corresponds a set of affine inequalities that are
satisfied by the points of the coadjoint orbit. Moreover the points
where the equalities hold are exactly the points of $\Mom(M_W)$.

\begin{defin}
  \label{def-W-datum}
  Let $W$ be a \tauc subspace and let $\data$ be a root datum. We say
  that $\data$ is a $W$-datum if $ \theta(\lieh) = \lieh $ and $
  \liebp \subset \liep_W$.  The set of $W$-data will be denoted by
  $\Data(W)$.
\end{defin}
By Corollary \ref{esistono-data} any \tauc subspace admits a
$W$-datum, which in general is not unique. If a $W$-datum is fixed,
then the following objects are well-defined: a subset $J\subset
\simple$ such that $\liep_W=\liep_J$, a subset $I\subset J$ such that
$W=V_I$, a highest weight $\mut \in \lieh^*$ and a line $x_0=[\vut]
\in M$ of highest weight vectors.  When we use these symbols we
understand that a $W$-datum has been chosen and that they refer to
that particular choice.  By \ref{W-I-peso} of Prop. \ref{W-I} $x_0\in
M_W$.  Any point $x_0\in M_W$ is obtained in this way from a
$W$-datum, but in general in many ways.  In other words, the map
$\data \mapsto x_0$ from the set of $W$-data to $M_W$ is surjective,
but in general non-injective.

\begin{lemma}\label{indep-W}
  Let $W \subset V $ be a $\tau$-connected subspace and let $\data$ be
  a $W$-datum.  Then the nonzero weights of the adjoint representation
  of $ \liez_\lieg(\lier_W)$ on $\liep_W$ coincide with the nonzero
  restrictions to $\liez_\lieg(\lier_W)$ of elements of $\rootp$.  In
  particular the set of these restrictions does not depend on the
  choice of $\data$, but only on $W$.
\end{lemma}
\begin{proof}
  By \ref{W-I-scomp} of Prop.  \ref{W-I} $\liez_\lieg(\lier_W) =
  \lieh_J$.  The decomposition
  \begin{gather*}
    \liep_W= \lier_W \oplus \bigoplus _{\alfa\in \root_+ \setminus
      \root_J}\lieg_\alfa
  \end{gather*}
  is clearly $\ad\, \lieh_J$-invariant.  For $\la\in \lieh_J^*$ denote
  by $U_\la$ the corresponding weight space of $\ad : \lieh_J \ra
  \gl(\liep_W)$.  If $\alfa \in \root$, then $\alfa\restr{\lieh_J} =0$
  if and only if $\alfa\in \root_J$. So $U_0=\lier_W$ and for $\la$ a
  nonzero weight $ U_\la = \bigoplus_{\alfa \in \root_+:
    \alfa\restr{\lieh_J} = \la} \lieg_\alfa$.
\end{proof}

Let $W$ be a \tauc subspace and let $\data$ be a $W$-datum. Let
$\liek'_W$ be as in \eqref{defKW}.  Set
\begin{gather*}
  \liec_W:=\liez_\liek(\lier_W) \oplus \liek'_W \qquad \liefw: =
  \lieh_I\cap \liek.
\end{gather*}
By definition $\liec_W$ depends only on $W$, while $\liefw$ does in
general depend also on $\data$.  Since $ \liec_W = (\lieh_J \cap
\liek) \oplus \liek_{I'}$ and $ \lieh_I = \lieh_J \oplus \lieh^{I'}$
by Prop. \ref{W-I} \ref{lieh-decomposition}, it follows that $\liefw
\subset \liec_W$ for any $W$-datum $\data$.
\begin{lemma}\label{caprette-nondip}
  If $x, x'\in M_W$, then $\Mom(x)\restr{\liec_W} =
  \Mom(x')\restr{\liec_W} $.
\end{lemma}
\begin{proof}
  $x'=a\cdot x$ for some $a\in K_W $.  $ [\liek_W, \liec_W]=0$ and
  $K_W$ is connected. So if $v\in \liec_W$, then $\Ad(a\meno)v=v$ and
  $ \deo \Mom(x'), v\ode = \deo \Mom(x),\Ad(a\meno) v\ode = \deo
  \Mom(x), v\ode $.
\end{proof}
Pick $x\in M_W$ and set
\begin{gather*}
  E_{W} : = \{ \la\in \liek^* : \la(v) = \deo \Mom(x) , v \ode \text {
    for all } v \in \liec_W\}.
\end{gather*}
By the previous lemma $E_{W}$ does not depend on the choice of $x\in
M_W$.  Similarly, for a $W$-datum $\data$ define
\begin{gather}
  \label{def-E-wdata}
  E_\wdata : = \{ \la\in \liek^* : \la(v) = \deo \Mom(x_0) , v \ode
  \text { for all } v \in \liefw\}.
\end{gather}
(As usual $x_0$ is the line thorough the highest weight vector
determined by $\data$.)
\begin{lemma}\label{dc-op}
  \begin{gather*}
    E_W = \bigcap_{\data \in \Data(W)} E_\wdata.
  \end{gather*}
\end{lemma}
\begin{proof}
  Since $\liefw \subset \liec_W$, $E_W \subset E_\wdata$. One
  inclusion follows. For the other take $\la\in \bigcap E_\wdata$.  If
  $v \in \liec_W$, then $v =v_1 + v_2\in \liez_{\liek}(\lier_W) \oplus
  \liek'_W$.  $v_2$ lies in the Lie algebra $\liet_2$ of a maximal
  torus of $\liek_W'$. Completing $ \liez_{\liek}(\lier_W) \oplus
  (\liet_2 \otimes \C)$ to a Cartan subalgebra of $\lieg$ we get a
  datum $\data$ such that $v\in \liefw$.  Since $\la \in 
  E_\wdata$ we get $ \la(v) = \deo \Mom(x_0), v \ode $. This proves
  that $\la \in E_W$.
\end{proof}

Next set
\begin{gather*}
  \liefwp := \{ v\in \liefw : i \alfa(v) \leq 0 \text { for all }
  \alfa \in \simple\}.
\end{gather*}
Note that $\liefwp = \{ v\in \liefw : i \alfa(v) \leq 0 \text { for
  all } \alfa \in \simple \setminus I\}$, since $\liefw =
\cap_{\alfa\in I} \ker i\alfa \cap\liek $.  Moreover $\{ i\alfa: \alfa
\in \simple\}$ is a basis of $\liet^*$ and $\{i \alfa\restr{\liefw} :
\alfa \in \simple \setminus I\}$ is a basis of $\liefw^*$.  Therefore
$\liefwp$ is just a (higher dimensional) quadrant, so in particular it
is a convex cone with nonempty interior spanned by a finite number of
rays.  Put
\begin{gather}
  \label{eq:def-CW}
  \begin{gathered}
    C_\wdata := \{ \la \in \liek^*: \la(v) \leq \deo \Mom(x_0), v \ode
    \text { for all } v\in \liefwp\}\\
    C_W :=\bigcap_{\data\in \Data (W)} C_\wdata \qquad
    \A:=\bigcup_{\data \in \Data (W)} \liefwp.
  \end{gathered}
\end{gather}
\begin{lemma}\label{rupicapra}
  Let $x$ be any point in $M_W$. Then $C_W = \{ \la\in \liek^* :
  \la(v) \leq \deo \Mom(x) , v \ode$ for all $ v \in \A\}$ and $E_W =
  \{ \la\in \liek^* : \la(v) = \deo \Mom(x) , v \ode \text { for all }
  v \in \A\}$.
\end{lemma}
\begin{proof}
  The choice of $x$ is irrelevant, by Lemma \ref{caprette-nondip}
  since $v\in \liec_W$.  The first statement follows directly from the
  definitions of $C_W$ and $\A$.  Moreover $\liefwp$ spans $\liefw$,
  so $ E_\wdata = \{ \la\in \liek^* : \la(v) = \deo \Mom(x) , v \ode
  \text { for all } v \in \liefwp\}$.  So the second statement follows
  from Lemma \ref{dc-op}. Put in another way, the linear span of $\A$
  is $\liec_W$.
\end{proof}

%
%
%
%

For $v\in \liek$ set $h_v :=\deo \Mom, v\ode$. $h_v$ is a component of
the moment map and it is the Hamiltonian function of $\xi_v$, i.e.
$-i_{\xi_v} \om = dh_v$.  Let $\htau$ be as in \eqref{rappresento}.
\begin{lemma}
  Let $W\subset V$ be $\tau$-connected and let $\data$ be a $W$-datum.
  \begin{enumerate}
  \item \label{boia-critico} If $v\in \liec_W$, then $x_0$ is a
    critical point of $h_v$.
  \item If $v\in \liec_W$, $u\in \liek$ and $w=\xi_u(x_0)$, then
    \begin{gather*}
      D^2h_v(x_0) (w,w) = i B([u,\htau],[u, v]).
    \end{gather*}
  \end{enumerate}
\end{lemma}
\begin{proof}
  Since $K$ is transitive on $M$ any $w\in T_{x_0}M $ can be written
  as $w=\xi_u(x_0)$ for some $u\in \liek$. Set $\alfa(t) =
  \exp(tu)\cdot x_0$ and $\ga(t) = \Ad(\exp(-tu))v$. Then
  \begin{gather*}
    h_v(\alfa(t)) = \deo \Mom ( \exp tu\cdot x_0) , v\ode = \deo \Mom
    ( x_0) , \Ad(\exp(-tu)) v\ode
    =\deo \Mom ( x_0) , \ga(t)\ode \\
    dh_v(x_0) (w ) = \desudtzero h_v(\alfa(t))= \deo \Mom(x_0),
    \dot{\ga}(0)\ode.
  \end{gather*}
  Clearly $ \dot{\ga}(0) = -[u, v]$, so using \eqref{eq:momento-x0}
  \begin{gather*}
    dh_v(x_0) (w ) = i \mut( [u,v]^\lieh)= B(i\htau, [u, v] ) = iB
    ([v,\htau], u).
  \end{gather*}
  Since $I'$ is orthogonal to $ \mut$, $\htau \in \lieh_{I'}$ and
  $\lieh_{I'} = \lieh_J \oplus \lieh^I \subset \lieh_J \oplus \liek_I
  $ by Prop \ref{W-I} \ref{lieh-decomposition}. Hence $[\liec_W,\htau
  ] =0$.  Thus $dh_v(x_0)=0$ and \ref{boia-critico} is proved.  To
  compute the Hessian consider again the path $\alfa$.  Since $x_0$ is
  critical
  \begin{gather*}
    D^2h_v(x_0) (w, w) = \dfrac {\mathrm{d^2}} {\mathrm{dt^2}} \bigg
    \vert _{t=0} h_v(\alfa(t)) = \deo \Mom(x_0), \ddot{\ga}(0)\ode
    = \deo \Mom(x_0), [u, [u, v]] \ode =\\
    =-iB ( \htau, [u, [u, v]] ) = iB ([u, \htau], [u, v]).
  \end{gather*}
  This proves the second statement.
\end{proof}
\begin{teo}
  \label{sticazzi}
  If $W$ is a proper \tauc subspace, then $ \OO=\Mom(M) \subset C_W $
  and $ \Mom\meno (E_W) = M_W$.
\end{teo}
\begin{proof}
  By the definition \eqref{eq:def-CW} of $C_W$, to prove the first
  part of the theorem we need to show that $ \OO=\Mom(M) \subset
  C_\wdata$ for any $W$-datum $\data$.  Let $\data$ be such a datum.
  It is enough to show that for any $v\in \liefwp$ the maximum of
  $h_v$ on $M$ is attained at $x_0$.  We know (see
  e.g. \cite[pp. 352ff]{knapp-beyond}) that it is possible to choose
  vectors $X_\beta \in \lieg_\beta$ for $\beta \in \root$ in such a
  way that as $\beta$ varies in $ \root_+$ the vectors
  \begin{gather}
    \label{base-compatta}
    A_\beta = \frac{i}{\sqrt{2}} (X_\beta + X_{-\beta}) \qquad B_\beta
    = \frac{1}{\sqrt{2}} (X_\beta - X_{-\beta})
  \end{gather}
  describe a $(-B)$-orthonormal basis of the orthogonal complement of
  $\liet = \liek \cap \lieh$ inside $ \liek$.  This means that
  \begin{gather}
    \label{eq:orthonormality}
    B(A_\beta, B_\ga) = 0 \qquad B(A_\beta, A_\ga) = B(B_\beta, B_\ga)
    = -\delta_{\beta\ga}.
  \end{gather}
  If $H\in \lieh$ then $ [A_\beta, H ]= -i\beta(H) B_\beta $ and $
  [B_\beta, H ]= i\beta(H) A_\beta$.  Given $w\in T_{x_0}M $, let
  $u\in \liek$ be such that $w=\xi_u(x_0)$. Write
  \begin{gather*}
    u= u^\liet + \sum_{\beta \in \root_+} ( a_\beta A_\beta + b_\beta
    B_\beta)
  \end{gather*}
  where $u^\liet \in \liet$.  Then for any $H\in \lieh$
  \begin{gather*}
    [u, H ] = i \sum_{\beta \in \root_+} \beta(H) \bigl ( - a_\beta
    B_\beta + b_\beta A_\beta \bigr).
  \end{gather*}
  Applying this formula with $H=\htau$ and then with $H=v\in
  \liefwp\subset \lieh$, using \eqref{eq:orthonormality} and applying
  the computation of the previous lemma, we get
  \begin{gather*}
    D^2h_v(x_0)(w,w) = i \sum_{\beta \in \root_+} (a_\beta ^2 + b_\beta^2 )
    \beta (H_\tau) \beta(v).
  \end{gather*}
  If $\beta \in \root_+$, then $\beta(\htau)\geq 0$ since $\mut$ is
  dominant, while $i\beta(v) \leq 0$ since $v\in \liefwp$. Therefore
  we get that the Hessian at $x_0$ is negative semidefinite.  But
  $h_v$ is a component of the moment map $\Mom : M \ra \liek^*$, hence
  it is a Morse-Bott function with critical points of even index (this
  is Frankel theorem, see e.g. \cite[Thm. 2.3,
  p. 109]{audin-torus-actions} or
  \cite[p. 186]{mcduff-salamon-symplectic}) and any local maximum
  point is an absolute maximum point (see
  e.g. \cite[p. 112]{audin-torus-actions}). Therefore $x_0$ is a
  global maximum of $h_v$ on $M$.  This proves that $\OO \subset
  C_\wdata$. Since $\data$ is an arbitrary $W$-datum, we get the first
  part of the theorem.

  Now fix a $W$-datum $\data$ and set
  \begin{gather*}
    F= \{x\in M: h_v(x) = h_v(x_0) \text{ for any } v\in \liefw\}.
  \end{gather*}
  Recalling definition \eqref{def-E-wdata} one immediately recognizes
  that
  \begin{gather}
    \label{F=baracca}
    F=\Mom\meno (E_\wdata).
  \end{gather}
  We start by showing that $F=M_W$.  Set for simplicity $\liet_I: =
  \liefw$.  The subgroup $ T_I=\exp \liet_I $ is a closed torus in $K$
  since $\liet_I\otimes \C =\lieh_I$ is a Cartan subalgebra of
  $\lies_I$.
  The moment map for the action of $T_I$ on $M$ is
  \begin{gather*}
    \Mom_I : M \ra \liet_I^* \qquad \Mom_I(x): =
    \Mom(x)\restr{\liet_I}.
  \end{gather*}
  Set $\la =\Mom_I(x_0)$.  It is immediate that $ F= \Mom_I \meno
  (\la)$.  By Atiyah theorem \cite[Thm. 1.A]{atiyah-commuting} every
  fibre of $\Mom_I$ is connected, so $F$ is connected.  Let $v_1,
  \lds, v_r$ be a basis of $\liefw$ such that each $v_j$ lies in
  $\liefwp$.  Each function $h_{v_j}$ attains its maximum at $x_0$,
  hence the point $\la$ is a vertex of $\Mom_I(M) \subset \liet_I^*$.
  By the Atiyah-Guillemin-Sternberg convexity theorem $F$ is a
  connected (symplectic) submanifold.  By Lemma \ref{caprette-nondip}
  we have $\deo \Mom(x), v\ode = \deo \Mom(x_0), v\ode $ for any $x\in
  M_W$ and any $v\in \liet_I$, so $\Mom_I(M_W)=\{\la\}$, i.e. $M_W
  \subset F$. Since $M_W$ and $F$ are compact connected submanifolds,
  to conclude that $F=M_W$ it is enough to check that $T_{x_0}F =
  T_{x_0}M_W$.  Denote by $P$ the stabilizer of $x_0$ and by $L=P\cap
  K$ its stabilizer in $K$.  Since $x_0$ is the line of highest weight
  vectors $\liebp\subset \liep$, so $P=P_E$ for some subset $E\subset
  \simple$. Let $\liem$ denote the orthogonal complement of $ \liel$
  inside $ \liek$. Then there is an isomorphism $\liem \cong T_{x_0}M$
  given by $v\mapsto\xi_v(x_0)$.  Since $F$ is the connected component
  of the fixed point set of $T_I$ through $x_0$, the tangent space
  $T_{x_0}F$ is the fixed set of the isotropy representation of $T_I$
  on $T_{x_0}M$.  Via the isomorphism $\liem \cong T_{x_0}M$ the
  isotropy representation is identified with the restriction of $\Ad
  T_I$ on $\liem$.  So $ T_{x_0} F = \liez(\liet_I) \cap \liem $.  We
  claim that $ T_{x_0} M_W = \liek_W \cap \liem $.  To check this it
  is enough to show that
  \begin{gather}
    \label{eq:scompo-liekw}
    \liek_W = \liek_W \cap \liel \oplus \liek_W \cap \liem.
  \end{gather}
  This follows from the following fact: if $T$ is the maximal torus
  corresponding to $\lieh$, then $\liek_W$ and $\liel$ are
  $T$-invariant subspaces of $\liek$ since $T\subset L$ and
  $\liek_W=\liek_I$. So also $\liem$ is $T$-invariant. Therefore both
  summands on the right in \eqref{eq:scompo-liekw} are $T$-invariant
  subspaces and $\liet \cap \liem=\{0\}$. Hence \eqref
  {eq:scompo-liekw} follows from the uniqueness of the real root
  decomposition. We have proved that $ T_{x_0} F = \liez(\liet_I) \cap
  \liem $ and $ T_{x_0} M_W = \liek_W \cap \liem $.  To show that
  $T_{x_0} F = T_{x_0} M_W $, we have to check that $\liez(\liet_I)
  \cap \liem=\liek_W\cap \liem$.  Recall \eqref{base-compatta} and set
  $ \liek_\beta = \R A_\beta + \R B_\beta $ and $ \liet^I=\lieh^I\cap
  \liek$.  Then
  \begin{gather*}
    \liek_W= \liek_I=\liet^I \oplus \bigoplus_{\beta \in \root_I\cap
      \root_+} \liek_\beta
    \\
    \liel = \liek \cap \liep_E = \liet \oplus \bigoplus_{\beta \in
      \root_E\cap \root_+} \liek_\beta\qquad \liem = \bigoplus_{\beta
      \in \root_+ \setminus \root_E} \liek_\beta
  \end{gather*}
  We claim that
  \begin{gather*}
    \liez(\lieh_I) = \lieh \oplus \bigoplus_{\beta \in \root_I}
    \lieg_\beta.
  \end{gather*}
  To prove this it is enough to check that if $\beta \in \root_+$ then
  $\beta \restr{\lieh_I} =0$ if and only if $\beta \in \root_I$. The
  collection $\{ \alfa\restr{\lieh_I} : \alf \not\in I\}$ is a basis
  of $\lieh_I^*$. If $\beta = \sum_{\alf \in \simple} m_\alfa \alf$
  and $\beta\restr{\lieh_I} =0$ then
  \begin{gather*}
    \sum_{\alfa \not \in I} m_\alfa \alfa\restr{\lieh_I} =0
  \end{gather*}
  so $m_\alf =0$ for any $\alfa \not \in I$ and $\beta \in
  \root_I$. The opposite implication is trivial, so the claim is
  proved.  From it we get
  \begin{gather*}
    \liez(\liet_I) = \liet \oplus \bigoplus_{\beta \in \root_I \cap
      \root_+} \liek_\beta\qquad \liez(\liet_I)\cap \liem =
    \bigoplus_{\beta \in \root_I\cap \root_+ \setminus \root_E}
    \liek_\beta = \liek_I \cap \liem.
  \end{gather*}
  Therefore $T_{x_0}M_W = T_{x_0}F$ and $F=M_W$. By \eqref{F=baracca}
  this means that $M_W=\Mom\meno(E_\wdata)$.  Using Lemma \ref{dc-op}
  we get
  \begin{gather*}
    M_W=\Mom\meno(\bigcap_\wdata E_\wdata) = \Mom\meno(E_W).
  \end{gather*}
  This concludes the proof of the theorem.
\end{proof}

We close this section by relating the moment map of $M_W$ to that of
$M$.  Let $W\subset V$ be $\tau$-connected. Denote by $\Mom^W : M_W
\ra \liek_W^*$ the moment map of the smaller flag manifold $(M_W,
\om)$.  Fix a $W$-datum $\data$ and write $ \htau = \zy + \zz $ where
$\zy\in i \liek_W$ and $\zz \in (i\liek_W) ^\perp$.
\begin{lemma}
  \label{zz-in-h}
  $\zz\in \lieh_I$.
\end{lemma}
\begin{proof}
  Set $ \liet = \lieh \cap \liek$, $ \liet^I = \lieh^I \cap \liek$, $
  \liet_I = \lieh_I \cap \liek$.  The two spaces $i\liet_I$,
  $i\liet^I$ are orthogonal to each other and $\htau \in i\liet = i
  \liet^I\oplus i\liet_I$. Moreover $ i\liet^I \subset i \liek_W$ and
  $i \liet_I \subset (i \liek_W)^\perp$.  Therefore $\zz\in i
  \liet_I$.
\end{proof}

Pick $x\in M_W$ and define a map $\thw : \liek_W^* \ra \liek^*$ by the
following recipe: if $\la\in \liek_W^*$ and $v\in \liek$, write $v=v_1
+ v_2$ where $v_1\in \liek_W$ and $v_2 \in \liek_W^\perp$ and set
\begin{gather}
  \label{eq:def-theta-W}
  \thw (\la) (v) := \la (v_1) + \deo \Mom(x) , v_2\ode .
\end{gather}
$\thw$ is an injective affine map.
\begin{lemma}
  \label{momento-ristretto}
  \begin{enumerate}
  \item \label{momento-ristretto-a} The map $\Theta_W$ does not depend
    on the choice of $x\in M_W$.
  \item \label{momento-ristretto-b} If $x\in M_W$, then $ \Mom(x) =
    \thw\bigl (\Momw(x) \bigr) $.
  \end{enumerate}
\end{lemma}
\begin{proof}
  Fix $x\in M_W$.  To prove \ref{momento-ristretto-a} it is enough to
  show that $\deo\Mom(x), v \ode = \deo \Mom(x_0) , v\ode$ for any
  $v\in \liek_W^\perp$.  There exists $a\in K_W$ such that $x=a \cdot
  x_0$ and $\deo \Mom(x), v\ode = \deo \Mom(x_0),\Ad(a\meno) v\ode $.
  So we need to show that $\deo \Mom(x_0) , \Ad(a\meno) v - v \ode =0$
  for any $a\in K_W$ and any $v\in \liek_W^\perp$.  Since $K_W$ is
  connected it is enough to prove that $ \Mom(x_0)\restr{[\liek_W,
    \liek_W^\perp]} =0$.  Note that $\Ad K_W$ preserves the
  decomposition $\liek=\liek_W \operp \liek_W^\perp$. By
  differentiating we infer that $[\liek_W, \liek_W^\perp] \subset
  \liek_W^\perp$.  So if $v\in \liek_W^\perp$ and $u\in \liek_W$, then
  $[u,v]\in \liek_W^\perp$ and $[u,\zz]=0$ by Lemma \ref{zz-in-h}.
  Therefore using \eqref{eq:momento-x0}
  \begin{gather*}
    \deo \Mom(x_0), [u,v]\ode = -iB(\htau, [u,v]) = -iB(\zz, [u,v]) =
    iB([u,\zz], v) =0.
  \end{gather*}
  This proves \ref{momento-ristretto-a}.  \ref{momento-ristretto-b} is
  simply a restatement of the fact that for any $x\in M_W$, $\Mom^W(x)
  = \Mom(x) \restr{\liek_W}$.
\end{proof}

\section{The Bourguignon-Li-Yau map}

\subsection{Definition and boundary behaviour}
\label{section-BLY-general}

In this section we introduce the Bourguignon-Li-Yau map in the most
general case and we investigate its boundary behaviour in the case of
a flag manifold.  Let $K$ be a connected compact Lie group (not
necessarily semisimple) and let $G=K^\C$ be its complexification.  Set
$X=G/K$, $\liek = \Lie K$, $\lieg = \Lie G$ and denote by $\theta$ the
Cartan involution.  Let $(M,J,g,\om)$ be a compact \Keler manifold and
assume that $G$ acts holomorphically and almost effectively on $M$
(i.e. only finitely many elements of $G$ act trivially on $M$.) Assume
also that the action of $K$ is Hamiltonian with moment map $\Mom: M\ra
\liek^*$.
\begin{defin}\label{def-bly}
  Given a probability measure $\mis$ on $M$ the
  \enf{Bourguignon-Li-Yau map} $\blym : X \ra \liek^*$ is defined by
  \begin{gather}
    \label{eq:def-blym}
    \blym(gK) = \int_M \Mom\bigl(\rho(g)\cdot x \bigr) d\mis(x)
  \end{gather}
  where $\rho(g)=\sqrt{g\theta(g\meno)}$ as in \eqref{eq:df-rho}.
\end{defin}
Recall the following notation: if $F: M \ra N$ is a measurable map
between measurable spaces and $\mu$ is a measure on $M$, then $F\pf
\mu$ denotes the push-forward or image measure.  This means that for a
measurable subset $E\subset N$ we have $F\pf \mu(E): =
\mu\bigl(F\meno( E)\bigr)$.  For any measurable function $\phi$ on $N$
the following change of variable formula holds true
\begin{gather}
  \label{eq:pushforward}
  \int_N \phi(y) d ( F\pf\mu)(y) = \int_M \phi(F(x)) d\mu(x)
\end{gather}
(see e.g.  \cite[p.163]{halmos},
\cite[p.221]{knapp-advanced-analysis}). In particular, if $M$ and $N$
are differentiable manifolds, $F$ is a diffeomorphism and $\mu$ is the
measure associated to a smooth top dimensional form $\eta$, then $
F\pf \mu $ is the measure associated to the form $(F\meno)^* \eta$.
\begin{lemma}\label{lemma-convesso}
  $\blym(X)$ is contained in the convex envelope of $\Mom(M)$ inside
  $\liek^*$.
\end{lemma}
\begin{proof}
  Let $F: M \ra M$ denote multiplication by $\rho(g)$.  By formula
  \eqref{eq:pushforward}
  \begin{gather*}
    \bly_\mis(gK) = \int_M \Mom(F(x)) d\mis(x) = \int_M \Mom(y) d
    (F\pf\mis)(y) = \int_{\Mom(M)} z \ d\left((\Mom\circ F \right) \pf
    \mis)(z).
  \end{gather*}
  $(\Mom\circ F) \pf \mis$ is a probability measure supported on
  $\Mom(M)$ and $\blym(gK)$ is its barycenter, which of course lies in
  the convex envelope of $\Mom(M)$.
\end{proof}
Let $G, K, \tau, \scalo$ be as at p. \pageref{data}, and let $M$ be
the associated flag manifold. Then $M$ is endowed with all the
structures needed to define the Bourguignon-Li-Yau map, so for any
probabily measure $\mis$ on $M$ there is a map $ \blym: X:=G/K \ra
\liek^* $ as in \eqref{eq:def-blym}.  Recall that $ \OO :=\Mom(M)
\subset \liek^* $ is a coadjoint orbit and that $\Mom: M \ra \OO$ is
an equivariant symplectomorphism. Let $\convo$ denote the convex hull
of $\OO$ and let $\intec $ denote the interior of $\convo$.
\begin{lemma}
  \label{intec-nonvuoto}
  (a) $0\in \intec$, (b) $\intec$ is a nonempty open convex domain,
  (c) $\overline{\intec}=\convo$, (d) $\convo$ is compact.
\end{lemma}
\begin{proof}
  (a) For $v\in \liek$ set $h_v=\deo \Mom, v\ode$.  Since $\liek =
  [\liek, \liek]$, $v=\sum_{i=1}^k[u_i,w_i]$ for some $u_i, w_i \in
  \liek$, so by the equivariance of the moment map
  \begin{gather*}
    h_v = \sum_{i=1}^k \deo \Mom, [u_i,w_i]\ode = \sum_{i=1}^k
    \xi_{w_i} h_{u_i} = \sum_{i=1}^k \{ h_{w_i}, h_{u_i}\}.
  \end{gather*}
  On the other hand, on a compact symplectic manifold Poisson brackets
  have zero mean:
  \begin{gather*}
    \int_M \{f,g\} \om^n = \int _M (X_f g) \om^n = \int_M L_{X_f} (
    g\om^n) = \int_M d\, i_{X_f} ( g\om^n) =0.
  \end{gather*}
  Hence any component $h_v$ of the moment map has zero mean.
  Therefore if $v\neq 0$ the function $h_v$ necessarily changes
  sign. This shows that $\OO$ is not contained in any halfspace with
  the origin on its boundary, which proves (a). (b) and (c)
  immediately follow.  To prove (d) it is enough to observe that $\OO$
  is compact.
\end{proof}

\begin{defin}
  We say that a probability measure $\mis$ on $M$ is \enf{
    $\tau$-admissible} if it does not charge the hyperplane sections
  of $M\subset \PP(V)$, i.e. if $\mis(H \cap M) =0$ for any hyperplane
  $H\subset \PP(V)$.
\end{defin}
This condition is of course satisfied by any measure that is
absolutely continuous with respect to the Riemannian measure. More
generally, if $Z\subset M$ is a complex submanifold such that
$Z\subset \PP(V)$ is full (i.e. $Z$ is not contained in any
hyperplane) and $\mis$ is a smooth measure on $Z$, then $\mis$, seen
as a measure on $M$, is $\tau$-admissible.

\begin{lemma}\label{convergenza-qo}
  Assume that $p_n \to p$ in $\XS$ and let $\ratp_n, \ratp$ be as in
  \eqref{eq:def-ratp}. If $\mis $ is $\tau$-admissible, then $ \ratp_n
  \to \ratp$ $\mis$-a.e.
\end{lemma}
\begin{proof}
  By Lemma \ref{prop-PV} (c) we can find unique $A_n, A\in \Herm(V)$
  such that $\tr A_n= \tr A=1$, $p_n=[A_n]$, $p=[A]$ and $A_n\to A$ in
  $\End V$.  Assume $p_n\in \XWn$ and $p\in \XW$. The set
  $N=\bigcup_n\PP(W_n^\perp)\cup \PP(W^\perp)$ has $\mis$-measure zero
  since $\mis$ is $\tau$-admissible.  By Lemma \ref{sqrt-cont}
  $\sqrt{A_n} \to \sqrt{A}$.  If $x=[v]\in M\setminus N$, then
  $\ratp_n(x)= [\sqrt{A_n} v] \to [\sqrt{A}v]=\ratp(x)$.  This
  concludes the proof.
\end{proof}

\begin{teo}
  \label{estendo}
  For any $\tau$-admissible probability measure $\mis$ on $M$ the map
  $\blym$ admits a continuous extension to $\XS$ that we still denote
  by $\blym$.  The extension is unique and satisfies $\blym(\XS)
  \subset \convo$.  If $p\in X_W$, then
  \begin{gather}
    \label{eq:blym-ext}
    \blym(p) = \int_M \Mom(\ratp (x)) d\mis (x).
  \end{gather}
\end{teo}
\begin{proof}
  Let $p\in \XW$.  By Lemma \ref{rat-map} \ref{sqrat-indeterminacy}
  the rational map $\ratp $ is defined outside $\PP(W^\perp)$. Since
  $\mis$ is $\tau$-admissible $M \setminus \PP(W^\perp)$ has full
  measure. By Lemma \ref{rat-map} \ref{ratiosu} $\ratp(M\setminus
  \PP(W)) \subset M_W \subset M$, so the function $x\mapsto
  \Mom(\ratp(x))$ is defined $\mis$-a.e. on $M$ and is bounded, hence
  integrable w.r.t. $\mis$. This shows that formula
  \eqref{eq:blym-ext} yields a well-defined map from $\XS$ to
  $\liek^*$.  If $W=V$ and $x=[v] \in M$, then $\ratp([v]) =
  [\sqrt{\tau(g) \tau(g)^*} v] = \rho(g)\cdot x$. Therefore
  \eqref{eq:blym-ext} agrees with \eqref{eq:def-blym} and yields an
  extension to $\XS$ of the map previously defined on $X$.  Let
  $\{p_n\} \subset \XS$ be a sequence converging to some point $p$.
  By Lemma \ref{convergenza-qo} $ \Mom\circ\ratp_n \to \Mom\circ
  \ratp$ $\mis$-a.e.  Since $\Mom$ is bounded on $M$ the integrals
  converge by the dominated convergence theorem. This shows that the
  map defined by \eqref{eq:blym-ext} is continuous on $\XS$.  The
  uniqueness of a continuous extension is obvious since $X$ is dense
  in $\XS$.  By Lemma \ref{lemma-convesso} $\blym(X) \subset \convo$,
  hence by continuity
  $\blym(\XS)\subset \convo$.
\end{proof}

Now fix a \tauc subspace $W\subsetneq V$.  If $\nu$ is a probability
measure on $M_W$ denote by
\begin{gather*}
  \bly_\nu^W : X_W \ra \liek_W^*
\end{gather*}
the Bourguignon-Li-Yau map of $M_W$ endowed with the structures
induced from $M$.
\begin{lemma}
  \label{lemma-traslo}
  For any $q \in X_W$,
  \begin{gather}
    \label{eq:blycmis}
    \blym(i_W(q)) = \thw\Bigl ( \bly^W_{\cmis} (q) \Bigr)
  \end{gather}
  where $\cmis : = (\Pii)\pf \mis$ is the push-forward probability
  measure on $M_W$ and $\Theta_W$ is defined in
  \eqref{eq:def-theta-W}.  If $\mis$ is the $K$-invariant probability
  measure on $M$, then $\cmis$ is the $K_W$-invariant probability
  measure on $M_W$.
\end{lemma}
\begin{proof}
  Assume $q=gK_W$ for $g\in S_W$. Then
  $p:=i_W(q)=[\tau_W(g)\tau_W(g)^* \oplus 0]\in \XS$.  Recall from
  Lemma \ref{rat-map} \ref{sqrat-5} that for any $x\in M \setminus
  \PP(W)$ we have $\ratp (x) = \rho(g) \cdot \Piw(x)$. Therefore
  \begin{gather}
    \label{eq:blymgtheta}
    \blym(p) = \int_M \Mom (\ratp(x)) d\mis (x) = \int_M \Mom \bigl
    (\rho(g)\cdot \Pii(x) \bigr) d\mis (x) = \int_{M_W} \Mom
    (\rho(g)\cdot y) d\cmis (y)
  \end{gather}
  Using Lemma \ref{momento-ristretto} we get
  \begin{gather*}
    \blym(p) = \int_{M_W} \Mom (\rho(g)\cdot y) d\cmis (y)= \int_{M_W}
    \thw \Bigl ( \Mom^W (\rho(g) \cdot y) \Bigr )
    d\cmis (y) = \\
    = \thw \biggl ( \int_{M_W} \Mom^W (\rho(g) \cdot y) d\cmis (y)
    \biggr ) = \thw \bigl ( \bly_{\cmis}^W(q) \bigr).
  \end{gather*}
  This proves \eqref{eq:blycmis}. To prove the last statement it is
  enough to take into account that the map $\Piw $ is
  $K_W$-equivariant.
\end{proof}
\begin{remark}
  If we use the scalar product $-B$ to identify $\liek^*$ with
  $\liek$, then formula \eqref{eq:blycmis} becomes simply $ \blym(p) =
  \bly_{\cmis}^W (p) + iZ_W$.  This means that the restriction of
  $\blym$ to any piece of the form $\XW$ is just a translation (by
  $iZ_W$) of the Bourguignon-Li-Yau map on the corresponding $M_W$
  with respect to a suitable measure.  In the case where $\ga$ is the
  $K$-invariant probability measure on $M$, this is particularly neat,
  since in that case $\cmis$ is just the $K_W$-invariant probability
  measure on $M_W$.
\end{remark}

\begin{prop}
  \label{bordo-bordo}
  Let $W$ be a proper $\tau$-connected subspace.  Then $\blym(\XS)
  \subset \ccim$ and
  \begin{gather}
    \label{eq:bb}
    \blym\meno (E_W) =\bigsqcup_{W'\subset W} \XWp = \overline{ \XW} .
  \end{gather}
\end{prop}
\begin{proof}
  By Theorem \ref{sticazzi} $\OO\subset \ccim$ so $\convo \subset
  \ccim$.  By Theorem \ref{estendo} $\blym(\XS) \subset \convo $ so
  $\blym(\XS) \subset \ccim$.  This proves the first part of the
  theorem.  Next we wish to show that $\blym(\XWp) \subset \ci$ if
  $W'\subset W$.  If $p\in \XWp$, then $p=gK_{W'}$ for some $g\in
  S_{W'}$ and by \eqref{eq:blymgtheta}
  \begin{gather*}
    \blym(p) = \int_{M_{W'}} \Mom(\rho(g) \cdot y) d\cmis(y).
  \end{gather*}
  Now, $\rho(g) M_{W'} \subset M_{W'} \subset M_W $, $\Mom(M_W)
  \subset \ci$ by Theorem \ref{sticazzi} and $\ci$ is obviously
  convex.  Hence $\blym(p) \in E_W$.  Conversely, fix $p\in \XWp$ for
  some \tauc subspace $W'\subset V$ and assume that $\blym(p) \in
  \ci$. We wish to prove that $W'\subset W$.  Fix a datum $\data$ and
  let $x_0$ be the line through the highest weight vector. Then by the
  definition of $E_W$ we have
  \begin{gather*}
    0= \deo \Mom(x_0), v\ode - \deo \blym(p) , v \ode = \\
    = \int_M \deo \Mom(x_0) - \Mom(\ratp(x)) , v \ode d\mis (x) =
    \int_M \Bigl ( h_v(x_0) - h_v(\ratp(x)) \Bigr ) d\mis (x).
  \end{gather*}
  Let $v_1, \lds, v_r$ be a basis of $\liet_I$ such that $v_j \in
  \liefwp$ for any $j$.  By Theorem \ref{sticazzi} $h_{v_j} \leq
  h_{v_j}(x_0)$, so $h_{v_j}(\ratp(x) ) = h_{v_j}(x_0)$
  $\mis$-a.e. for any $j$.  For any $j$ there is a subset $E_j\subset
  M$ such that $h_{v_j} \circ\ratp = h_{v_j}(x_0)$ on $ M\setminus
  E_j$ and $\mis(E_j)=0$.  Set
  \begin{gather*}
    E:= \cup_{j=1}^n E_j \cup \PP(W'^\perp).
  \end{gather*}
  Then $\mis(E)=0$. By linearity $h_v\circ \ratp = h_v(x_0)$ on $
  M\setminus E$ for any $v\in \liet_I$. By Theorem \ref{sticazzi} this
  means that $\ratp\left( M\setminus E \right) \subset M_W\subset
  \PP(W)$.  If $p = i_{W'}(gK_W)$, then $\ratp = \rho(g) \circ
  \check{\pi}_{W'} = \check{\pi}_{W'} \circ \rho(g) $. Set $F =
  \rho(g)(E)$.  Then $\check{\pi}_{W'} (M-F) = \ratp(M-E) \subset
  \PP(W)$.  We claim that the projective subspace generated by $
  \check{\pi}_{W'} (M-F) $ coincides with $\PP(W')$.  Of course $
  \check{\pi}_{W'} (M-F) \subset \PP(W')$, so the subspace generated
  is contained in $\PP(W')$.  On the other hand if $U\subset W'$ is a
  subspace such that $ \check{\pi}_{W'}(M\setminus F) \subset \PP(U)$,
  then
  \begin{gather*}
    M \setminus F \subset \PP(U + (W')^\perp) \qquad
    \rho(g)\meno F = E \\
    M - E = \rho(g)\meno ( M- F)\subset \PP((\rho(g))\meno ( U +
    (W')^\perp)).
  \end{gather*}
  Since $\mis(M\setminus E)=1$ we get
  \begin{gather*}
    \rho(g)\meno ( U + (W')^\perp)=V\qquad U + (W')^\perp =V.
  \end{gather*}
  Since $U\subset W'$ this yields $ U=W'$. The claim is proved.  Now
  from $\check{\pi}_{W'} (M-F) \subset \PP(W)$ we get $\PP(W') \subset
  \PP(W)$ so $W'\subset W$.  This proves the first equality in
  \eqref{eq:bb}.  The second one follows from Satake analysis of the
  boundary components.
\end{proof}

It is useful to recall a few definitions and results regarding convex
bodies (see e.g. \cite{schneider-convex-bodies}). If $V$ is a real
vector space and $E\subset V$ the \emph{relative interior} of $E$,
denoted $\relint E$ is the interior of $E$ in its affine hull. If $E$
is a convex set in $V$ a \emph{face} of $E$ is a convex subset
$F\subset E$ with the following property: if $x,y\in E$ and
$\relint[x,y]\cap F\neq \vacuo$ then $[x,y]\subset F$.  The
\emph{extreme points} of $E$ are the points $x\in E$ such that $\{x\}$
is a face.  If $E$ is compact the faces are closed
\cite[p. 62]{schneider-convex-bodies}.  If $F$ is a face of $E$ we say
that $\relint F$ is an \emph{open face} of $E$.
\begin{lemma}
  \label{facciotta}
  Let $V$ be a finite-dimensional real vector space and let $E\subset
  V$ be a compact convex set. Let $\{\la_\alf\} _{\alf\in A} $ be a
  family of linear functionals on $V$ and let $a_\alf \in \R $ be such
  that $E\subset \{v\in V: \la_\alf (v) \leq a_\alf $ for any $\alf\in
  A \}$. Then
  \begin{gather*}
    F:=\{v\in E: \la_\alf(v) = a_\alf \text { for any } \alf \in A\}
  \end{gather*}
  is empty or a face of $E$.
\end{lemma}
\begin{proof}
  Fix $\alf \in A$ and $x, y \in E $. If $\relint [x,y]\cap F\neq
  \vacuo$, the affine function $\la_\alf$ takes its maximum at an
  interior point of $[x,y] $. Therefore it must be constant on
  $[x,y]$.
\end{proof}

\begin{teo}
  [\protect{\cite[p. 62]{schneider-convex-bodies}}] If $E$ is a
  compact convex set and $F_1,F_2$ are distinct faces of $E$ then
  $\relint F_1 \cap \relint F_2=\vacuo$. If $G$ is a nonempty convex
  subset of $ E$ which is open in its affine hull, then $G
  \subset\relint F$ for some face $F$ of $E$. Therefore $E$ is the
  disjoint union of its open faces.
\end{teo}

\begin{lemma}\label{relinquo}
  Let $V$ be a finite-dimensional real vector space, $E\subset V$ a
  compact convex set and $A\subset E$ a convex set. Assume that $A$ is
  dense in $E$ and is open in its affine hull. Then $A=\relint E$.
\end{lemma}

\begin{lemma}
  \label{spacco-la-faccia}
  \begin{enumerate}
  \item[(a)] For any \tauc subspace $W$ the set $ F_W:=\convo\cap \ci
    $ is a face of $\convo$.
  \item[(b)] If $W$ and $W'$ are $\tau$-connected subspaces of $V$ and
    $F_W\subset F_{W'} $, then $W\subset W' $.
  \item[(c)] $\relint F_W = \relint F_{W'}$ if and only if $W=W'$.
  \end{enumerate}
\end{lemma}
\begin{proof}
  (a) It follows from Theorem \ref{sticazzi} that $\convo \subset
  C_W$. By Lemma \ref{rupicapra} $E_W$ is the affine subspace where
  the inequalities defining $C_W$ become equalities.  Hence Lemma
  \ref{facciotta} ensures that $F_W$ is a face of $\convo$.  (b) If $
  F_W\subset F_{W'}$ then $\ci\cap\OO \subset \cip \cap\OO$. Using
  again Theorem \ref{sticazzi} we get that
  \begin{gather*}
    M_W= \Mom\meno(\ci) \subset \Mom\meno(\cip) = M_{W'}.
  \end{gather*}
  By Corollary \ref{MW-determina} this implies that $W\subset W'$.  (c) If
  $\relint F_W = \relint F_{W'}$ also $F_W=F_{W'}$ by taking the
  closures.  Hence $W=W'$ by (b).
\end{proof}

\begin{teo}
  \label{borda}
  Let $\mis$ be a $\tau$-admissible measure.  Then $\blym (\partial
  \XS) \subset \partial \convo$.
\end{teo}
\begin{proof}
  By Theorems \ref{Satakone} and \ref{ziasti} the boundary of $ \XS$
  is the disjoint union of $ \XW$ as $W$ varies over all proper \tauc
  subspaces. By Proposition \ref{bordo-bordo} for any such $W$, $\blym
  (\XW) \subset E_W\cap \convo=F_W$ and by Lemma
  \ref{spacco-la-faccia} $F_W$ is a proper face of $\convo$, so
  $F_W\subset \partial \convo$.
\end{proof}

\subsection{The case of the $K$-invariant measure}
\label{section-K-inv}

Let $G, K, \tau. \scalo$ be as at p. \pageref{data}, and let $M$ be
the associated flag manifold.  Let $\mu$ denote the unique
$K$-invariant probability measure on $M$, i.e.  the Riemannian volume
normalized to be of unit mass. Set
\begin{gather*}
  \bly:=\bly_\mu.
\end{gather*}

\begin{teo}
\label{infiammata}
  $ \bly\restr{ X} $ is a diffeomorphism of $X$ onto the interior of
  $\convo$.
\end{teo}
\begin{proof}
Consider the map
\begin{gather*}
  F: G\ra \liek^* \qquad F(g) :=\int_M \Mom(g\cdot x)d\mu(x).
\end{gather*}
Fix $g\in G$, $v\in \liek$ and set $\tmu = g\pf \mu$.  Denote by $R_g
: G \ra G$ the right translation and set $w:=(dR_g)(e)(iv)\in T_g G$.
\begin{gather*}
  dF(g)\left ( w \right ) = \desudtzero F( \exp(itv) g) = \desudtzero
  \int_M \Mom(\exp(itv \cdot y ) d\tmu(y)
  \\
  \deo dF(g)\left (w \right ) , v \ode = \int_M
  \desudtzero \deo \Mom(\exp(itv \cdot y ) , v \ode d\tmu(y)\\
  \desudtzero \deo \Mom(\exp(itv \cdot y ) , v \ode = d \deo \Mom, v
  \ode(y) ( \xi_{iv}(y))= - i_{\xi_v} \om (J\xi_v(y)) = - |\xi_v(y)|^2\\
  \deo dF(g)\left ( w \right ) , v \ode = - \int_M |\xi_v|^2 d\tmu.
\end{gather*}
Since $\tmu$ is a smooth measure with strictly positive density, if
$dF(g)(w)=0$, then $\xi_v \equiv 0$ and $v=0$. This shows that the
restriction of $dF(g)$ to $( dR_g)(e)(i\liek) \subset T_gG$ is an
isomorphism. In particular $dF(g)$ is onto and $F$ is a submersion.
Let $g=\rho(g) a$ be the polar decomposition, i.e.  $a\in K$.  Since
$a\pf \mu=\mu$
\begin{gather*}
  F(g) :=\int_M \Mom(g \cdot x)d\mu(x)= \int_M \Mom(\rho(g)\cdot
  x)d(a\pf\mu)(x)= \int_M \Mom(\rho(g)\cdot x)d\mu(x)= \\
  =\bly( gK).
\end{gather*}
So if $\pi: G \ra X$ is the canonical map $\pi(g) = gK$, then
$\bly\restr{X} \circ \pi = F$.  Since $F$ is a submersion,
$\bly\restr{X}$ is a local diffeomorphism and in particular an open
map.  So $\bly(X)$ is open in $\liek^*$.  Since $\bly(X) \subset
\convo$ by Lemma \ref{lemma-convesso}, it follows that $\bly(X)
\subset \inte \convo = \Omega$.  On the other hand, by Theorem
\ref{borda}, $\bly(\partial \XS) \subset \partial \convo$.  Therefore
$\bly\restr{X} : X \ra \Omega$ is also a proper map. It follows that
it is a covering of $\Omega$ and therefore a diffeomorphism.
\end{proof}

\begin{teo}
  \label{omeo-K}
    \begin{enumerate}
    \item [(a)] The Bourguignon-Li-Yau map $\bly$ is a homeomorphism
      of $\XS$ onto $\convo$.
  \item [(b)] For any \tauc subspace $W$, $\bly(\XW) =\relint \FW$ and
    $ \bly \circ i_W $ is a diffeomorphism of $X_W$ onto $\relint
    F_W$.
  \end{enumerate}
\end{teo}
\begin{proof}
  By Theorem \ref{estendo} the map $\bly: \XS\ra \liek^* $ is
  continuous and its image is contained in $\convo$.  By Theorem
  \ref{borda} $\bly( \partial \XS) \subset \partial \convo$.  Since
  $\XS$ is compact and $\bly(\XS) \supset \bly(X) = \intec$,
  $\bly(\XS) = \convo$, so $\bly$ is surjective.  Next consider a
  proper $\tau$-connected subspace $W$.  By Proposition
  \ref{bordo-bordo} $\bly(\XW ) \subset \convo \cap \ci$.  By Lemma
  \ref{lemma-traslo} the diagram
  \begin{gather}
    \label{diagrammone}
    \begin{diagram}
      \node[1] {X_W} \arrow[2]{e,t}{\bly^W}
      \arrow[1]{s,l}{i_W} \node[2]{\liek_W^*} \arrow{s,r}{\thw}
      \\
      \node[1] {\XW} \arrow[2]{e,t}{\bly }\node[2] {\liek^*}
    \end{diagram}
  \end{gather}
  commutes.  (Here $\bly^W$ is the Bourguignon-Li-Yau map of $M_W$
  with the invariant probability measure.)  Set $\OO_W : = \Mom^W
  (M_W) \subset \liek_W^*$, $\convo_W : = $ convex hull of $ \OO_W$
  and $\Omega_W := \inte \convo_W$.  By what we have just proved
  applied to the flag manifold $M_W$, $\bly^W(\overline{X}_W) =
  \convo_W$ and $\bly^W\restr{X_W}$ is a diffeomorphism onto
  $\intec_W$.  Since $\thw $ is an affine embedding, it follows that
  $\bly\circ i_W$ is a diffeomorphism of $X_W$ onto $\thw(\intec_W)$.
  We wish to show that
  \begin{gather}
    \label{eq:Theta-leccamelo}
    \thw(\intec_W)=\relint \left ( \ci \cap \convo \right ).
  \end{gather}
  Let $\overline{X}_W$ denote the Satake compactification of $X_W$
  associated to $\tau_W$ and let $\overline{\XW}$ denote the closure
  of $\XW$ in $\XS$.  If in \eqref{diagrammone} we replace $X_W$ by
  $\overline{X}_W$ and $\XW$ by $\overline{\XW}$, then by continuity
  the diagram still commutes.  Since $\bly$ is surjective, Proposition
  \ref{bordo-bordo} implies that $\bly(\overline{\XW}) = \ci \cap
  \convo$, hence
  \begin{gather*}
    \thw\circ \bly^W (\overline{X}_W) = \bly(\overline{\XW})
    = \ci \cap \convo.
  \end{gather*}
  So $\thw(\intec_W) =\thw\circ \bly^W (X_W) $ is dense in
  $\ci\cap\convo$.  As $\intec_W$ is open in $\liek_W^*$,
  $\thw(\intec_W) $, is an open convex domain in $\thw(\liek_W^*)$
  which is its affine hull.  Lemma \ref{relinquo} implies that
  \eqref{eq:Theta-leccamelo} holds true and concludes the proof of
  (b).  Now we wish to prove that $\bly$ is injective on the
  boundary. By the above it is injective on $X$ and also on each
  boundary component $\XW$. Moreover each boundary component $\XW$ is
  mapped to $\relint F_W$.  If $\XW \neq \XWp$, then $W\neq W'$ hence,
  by Lemma \ref{spacco-la-faccia} (c), $\relint F_W \neq \relint
  F_{W'}$. But $F_W$ and $F_{W'}$ are faces of $ \convo$, by Lemma
  \ref{spacco-la-faccia} (a), hence $\relint F_W \cap \relint
  F_{W'}=\vacuo$. It follows that $\bly$ is injective. Since $\XS$ is
  compact and $\convo$ is Hausdorff $\bly$ is a homeomorphism. This
  finally proves (a).
\end{proof}

\begin{teo}
  \label{grado}
  Let $\mis$ be a $\tau$-admissible measure.  Then the
  Bourguignon-Li-Yau map is $\blym : \XS\ra \convo$ is surjective.
\end{teo}
\begin{proof}
  Set $\mis_t:=(1-t)\mu + t\mis$ and
  \begin{gather*}
    H: \XS \times [0,1] \ra \convo \qquad H(p,t) := \bly_{\mis_t} (p).
  \end{gather*}
  $\mis_t$ is $\tau$-admissible for every $t\in [0,1]$, $H$ is
  continuous and $H(\partial\XS \times [0,1] ) \subset \partial
  \convo$ by Theorem \ref{borda}. Since $\bly\restr{\partial\XS} =
  H(\cdot, 0) \restr{\partial\XS}$ is a homeomorphism, it has degree
  $1$. Hence the same holds for $H(\cdot, 1)
  \restr{\partial\XS}=\blym\restr{\partial\XS} $. By a classical
  topological argument this yields the surjectivity of $H(\cdot, 1) =\blym$.
\end{proof}

\subsection{Furstenberg compactifications}
\label{section-furst}

Another way to compactify $X=G/K$ was found by Furstenberg
\cite{furstenberg-Poisson} in his search for an analogue of the
Poisson formula for the unit disc.  We recall very briefly the
definition in the case of type IV symmetric spaces (see \cite[\S
I.6]{borel-ji-libro} for the general case).  Let $G$ be a connected
complex semisimple Lie group and $K$ a maximal compact subgroup. A
homogeneous space $M=G/P$ is called a \emph{Furstenberg boundary} of
$G$ if for every probability measure $\nu$ on $M$, there exists a
sequence $g_j \in G$ and a point $x\in M$ such that ${g_j}\pf \nu
\rightharpoonup \delta_x$. Using Iwasawa structure theory, Moore
\cite[Thm. 1]{moore-compactifications} proved that $G/P$ is a boundary
if and only if $P$ is parabolic. In this case $K$ acts transitively
and $M$ has a unique $K$-invariant probability measure $\mu$.  For any
topological space $Z$, let $\proba(Z)$ denote the set of Borel
probability measures on $Z$ provided with the weak topology.  Since
$\mu$ is $K$-invariant the map
\begin{gather*}
  G \ra \proba(M) \qquad g \mapsto g\pf \mu
\end{gather*}
descends to a continuous map $i_{M} : X=G/K \ra \proba(M)$, which is
injective iff $P$ does not contain simple factors of $G$ (see
\cite[Thm. 4]{moore-compactifications} or
\cite[Prop. I.6.16]{borel-ji-libro}).  In this case $M$ is called a
\enf{faithful} Furstenberg boundary and the set
\begin{gather}
  \label{eq:def-furstenberg}
  \XF: = \overline {i_{M} (X)}
\end{gather}
is called the \emph{Furstenberg compactification} of $X$ associated to
the the faithful boundary $M$.  Fix an irreducible complex
representation $\tau : G \ra \Gl(V)$ such that $P$ is the stabilizer
of some $x_0\in \PP(V)$ and $\ker \tau$ is finite. Such
representations always exist. Then $M$ can be identified with the
orbit $G\cdot x_0$, which is the unique closed orbit in $\PP(V)$.  If
a $K$-invariant Hermitian product is fixed on $V$, $M$ becomes a flag
manifold, so we are back to the previous setting.
\begin{teo}
  \label{furst}
  The map
  \begin{gather*}
    \Ga : \XS \ra \XF \qquad \Ga(p) := \ratp\pf \mu
  \end{gather*}
  is a $G$-equivariant homeomorphism of $\XS$ onto $\XF$ such that
  $i_M=\Ga \circ i_\tau$ (compare \eqref{eq:imbeddoinpos}).
\end{teo}
\begin{proof}
  For any $p\in \XS$ the rational map $\ratp: M \dashrightarrow M $ is
  defined $\mu$-a.e., since $\mu$ does not charge linear sections of
  $M \subset \PP(V)$. Therefore $\Ga(p) = \ratp\pf \mu$ is
  well-defined for any $p\in \XS$.  If $p_ n \to p$, then $\ratp_n \to
  \ratp$ $\mu$-a.e.  by Lemma \ref{convergenza-qo}.  It follows that
  $\Ga(p_n) \rightharpoonup \Ga(p)$. So $\Ga$ is continuous.  Moreover
  $\Ga\circ i_\tau = i_M$ (see \eqref{eq:imbeddoinpos} for the
  definition of $i_\tau$). Indeed by Lemma \ref {rat-map}
  \ref{sqrat-5}, if $p=i_\tau (gK) $, then $\ratp = L_{\rho(g)} : M
  \ra M $ so $\Ga (p) = \rho(g)\pf \mu = g\pf \mu = i_M (gK)$.
  Therefore $\Ga(M)=i_M(X)$ is dense in $\XF$. Since $\XS$ is compact
  it follows that $\Ga$ is surjective. Consider now the following maps
  \begin{gather*}
    F : \proba(M) \ra \proba(\convo) \qquad F(\mis) = \Mom\pf \mis\\
    \BB : \proba(\convo) \ra \convo \qquad \BB(\nu) = \int_{\convo}
    zd\nu(z).
  \end{gather*}
  $\BB(\nu)$ is just the barycenter of $\nu$. Now $\BB \circ F \circ
  \Ga = \bly$.  Since $\bly$ is a homeomorphism, $\Ga$ is
  injective. As $\proba(M)$ is a Hausdorff space, $\Ga$ is a
  homeomorphism.  Moreover $i_M=\Ga \circ i_\tau$, so $\Ga$ is
  $G$-equivariant on $X$. By continuity it is equivariant also on the
  compactifications.
\end{proof}

\section{Application to eigenvalue estimates}
\label{sezione-autovalore}

In this section we apply Theorem \ref{grado} to a problem in spectral
geometry. Let $M$ be a complex manifold and $g$ a \Keler metric.
Denote by $\Delta_g : \cinf(M) \ra \cinf(M)$ the Laplace-Beltrami
operator
\begin{equation*}
  \Delta_g f = -d^* d f = 
  2 \  g^{i\bar{j}}\frac{\partial^2 f}{\partial z^i \partial \bar{z}^j}.
\end{equation*}
It is well-known that $-\Delta_g$ is a positive elliptic operator with
eigenvalues $0 \leq \la_1(g) \leq \la_2(g) \leq \cdots$.  Assume that $M$ is a
Fano manifold and $g_{KE}$ is a \KE metric with \Keler class $\om_{KE}
\in 2\pi\chern_1(M)$ (i.e. $\Ric(g_{KE}) = g_{KE}$).  Denote by
$\liek$ the Lie algebra of Killing vector fields and by $\lieg$ the
algebra of holomorphic vector fields (considered as real fields). The
map $v\mapsto Jv$ endows $\lieg$ with the structure of complex Lie
algebra.  For $u \in \cinf(M,\R)$ let $X_u$ denote the Hamiltonian
vector field such that $du = -i_{X_u}\om_{KE}$.

\begin{teo}
  \label{futaki}
  Let $M$ be a Fano manifold and let $g_{KE}$ be a \KE metric as
  above.  Then (a) $\liek$ is a real form of $\lieg$; (b) if $\lieg
  \neq \{0\}$ then $\la_1(g_{KE}) =2$. (c) Set $\Lambda_1 =\{u\in
  \cinf(M,\R):- \Delta_{KE} u =2u\}$.  If $u \in \Lambda_1$, then
  $X_u\in \liek$. (d) Define $ L : \Lambda_1 \ra \liek $ by $ L(u) =
  X_u $. Then $L$ is an isomorpism and the map $\Mom: M \ra \liek^*$
  defined by
  \begin{gather*}
    \deo \Mom(x), v\ode := L\meno(v) (x)
  \end{gather*}
  is a moment map for the action of $\operatorname{Isom}(M, g_{KE})$.
\end{teo}
This statement combines results of Matsushima and Futaki, see
\cite[p. 95]{kobayashi-trans}, \cite[Thm.  2.4.3, p.41]{futaki-libro},
\cite[Lemma 4.2]{futaki-ridotto}.

Assume now that $M$ is a Hermitian symmetric space of the compact
type.  We are interested in upper bounds for $\la_1(g)$ that hold for
any \Keler metric in $2\pi\chern_1(M)$.  Using the notation
\eqref{iaap} of the introduction, we wish to estimate
$I(2\pi\chern_1(M))$.  Let $g_{KE}$ be the symmetric \KE metric with
$\Ric(g_{KE}) = g_{KE}$ and let $\om_{KE}$ be its \Keler form.  Let
$K$ denote the connected component of the identity of
$\operatorname{Isom}(M,g_{KE})$. $K$ acts transitively on $M$.  Set
$n:=\dim_\C M$.
\begin{lemma}
  \label{lemmazk}
  Let $e_1, \lds, e_l$ be an orthonormal basis of $\liek:=\Lie \cK$
  with respect to the scalar product $-B/2$ ($B$ is the Killing form).
  Set $f_j=\deo \Mom, e_j\ode$ and
  \begin{gather*}
    \alfa : = \sum_{j=1}^l i \partial f_j \wedge \debar f_j .
  \end{gather*}
  Then $ \alfa = \om_{KE} $ and $ |\Mom|^2 =n$.
\end{lemma}
\begin{proof}
  Fix a point $x_0\in M$ and let $K_0$ be its stabilizer.  If $\liek =
  \liek_0 \oplus \liem$, then $\liem$ is $\ad\ K_0$-invariant and we
  identify it with $T_{x_0}M$ via the mapping $v\mapsto \xi_v(x_0)$.  By
  \cite[p. 261]{KN-I} there is $H\in \liez(\liek_0)$ such that $J_{x_0} =
  \ad H\restr{\liem}$.  The \KE metric $g_{KE}$ is the one associated
  with the scalar product $\hh=- \demi B$ on $\liem$, see
  \cite[p. 262]{KN-I}.  Define a map $ \Mom: M \ra \liek^*$ by
  \begin{gather}
    \deo \Mom(a \cdot x_0), v \ode = \frac{1}{2} B ( H, \Ad(a\meno) v
    ).\label{momzk}
  \end{gather}
  This map is well defined since $H \in \liez(K_0)$. Moreover $\Mom$
  is $K$-equivariant.  We claim that $\Mom$ is the moment map of $(M,
  \om_{KE})$. Indeed for $w\in T_{x_0}M$
  \begin{gather}
    \label{zk}
    \begin{gathered}
      d \deo \Mom, v\ode(x_0)( w) = \desudtzero \deo \Mom ( e^{tw}
      \cdot x_0) ,
      v\ode = \demi \desudtzero  B(H,\Ad(  e^{-tw}) v)  = \\
      = - \demi B(H, [w,v]) =-\demi B ([v, H], w).
    \end{gathered}
  \end{gather}
  Let $ v= v_1 + v_2 $ with $v_1\in \liek_0$, $ v_2 \in \liem$. Then $
  [v, H] = [v_2, H]$ and $\xi_v(x_0) = v_2$, so
  \begin{gather*}
    d \deo \Mom, v\ode(x_0)( w) = \demi B([H, v_2], w) = -g_{KE}
    (J_{x_0} v_2, w) = -\om_{KE} (\xi_{v_2}(x_0), w) =\\
    = - i_{\xi_v}\om_{KE}(w).
  \end{gather*}
  This shows that $d\deo \Mom, v\ode = - i{\xi_v}\om_{KE}$ at
  $x_0$. By $K$-equivariance this holds on all of $M$, so $\Mom$ is
  the moment map.  A simple computation shows that for $f\in \cinf(M,
  \R)$ and tangent vectors $X,Y$
  \begin{gather*}
    i\partial f \wedge \debar f (X,Y) = \demi ( Yf\cdot JXf - Xf \cdot
    JYf).
  \end{gather*}
  Apply this to $X=w_1, Y=w_2$ and $f=\deo \Mom , v \ode$, where $w_i
  \in \liem$ and $v\in \liek$.  Using \eqref{zk} and $[H, [H,w_i]] =
  J^2 w_i = -w_i$ we get
  \begin{gather*}
    i\partial f \wedge \debar f (w_1, w_2) = \\
    = \demi \biggl ( \demi B\left (H, [v,w_2] \right) \demi B (H, [v,[H,w_1]]) -
    \demi B(H, [v,w_1]) \demi B (H, [v,[H,w_2]]) \biggr )
    =\\
    =\demi \biggl ( \hh ([H,w_2], v) \hh (v, [H,[H,w_1]]) - \hh
    ([H,w_1], v) \hh (v, [H,[H,w_2]]) \biggr ) = \\
    =\demi ( \hh ([H,w_1], v) \hh (v,w_2) - \hh ([H,w_2], v) \hh
    (v,w_1) ) .
  \end{gather*}
  Since $e_j$ is an orthonormal basis of $(\liek, \hh)$
  \begin{gather*}
    \alfa(w_1, w_2) = \demi \sum_{j=1}^l \biggl ( \hh ([H,w_1], e_j)
    \hh (e_j, w_2) - \hh ([H,w_2], e_j) \hh (e_j, w_1) \biggr )=
    \\
    = \demi \biggl ( g_{KE} (Jw_1, w_2)- g_{KE} ( Jw_2, w_1) \biggr )
    = \om_{KE}(w_1, w_2).
  \end{gather*}
  This proves that $\alfa=\omega_{KE}$.  The norm $|\Mom|$ is constant
  since $K$ is transitive.  Using \eqref{momzk} we get
  \begin{gather*}
    |\Mom|^2(x_0)= \sum_{j=1}^l \deo \Mom, e_j \ode^2 = \sum_j \bigl [
    \demi B ( H,e_j) \bigr ] ^2 = \sum_j [ \hh(H, e_j)]^2 = \\ = \hh
    (H, H) = -\demi B(H,H).
  \end{gather*}
  But $(\ad H)^2$ is $0$ on $\liek_0$ and is $-1$ on $\liem$. Since
  $\dim_\R \liem = 2n$, $ B(H,H) = -2n$ and $|\Mom|^2=n$.
\end{proof}
\begin{teo}
  \label{teo-sotto}
  Let $M$ be a Hermitian symmetric space of the compact type. Let $N$
  be a \Keler manifold and let $a\in H^2(N)$ be a \Keler class.  Let
  $F: N \ra M$ be a holomorphic map and assume that no nontrivial
  section of $-K_M$ vanishes on $F(N)$.  Let $n=\dim_C M$ and
  $d=\dim_C N$. Then for any K\"ahler metric $g$ with K\"ahler form
  $\om \in a$
  \begin{gather}
    \label{stima}
    \la_1(g) \leq \frac{4 \pi d}{n} \cdot \frac{ \int_N
      F^*(\chern_1(M) ) \cup a^{d-1}} { \int_N a^{d}}.
  \end{gather}
\end{teo}
\begin{proof}
  Let $g_{KE}$, $\om_{KE}$, $K$, $\liek$ and $\lieg$ have the same
  meaning as above. Let $G$ denote the connected component of the
  identity of the automorphism group of $M$. By Theorem \ref{futaki}
  $K$ is a maximal compact subgroup of $G$.  Set $X:=G/K$.  Since $K$
  acts transitively the norm $|\Mom|$ is constant. Therefore, using
  the notation of Lemma \ref{lemmazk} and setting $l:=\dim K$,
  \begin{gather*}
    \sum_{j=1}^l \int_M f_j^2 \frac{\om_{KE}^n}{n!} = |\Mom|^2
    \vol(M,g_{KE}).
  \end{gather*}
  The anticanonical bundle $-K_M$ is very ample.  Set $V:=H^0 (M,
  -K_M)^*$.  Since $-K_M$ is $G$-equivariant there is an obvious
  representation $\tau : G \ra \Gl(V)$.  The metric $g_{KE}$ induces
  an $L^2$-scalar product $\scalo$ on $V$ which is $K$-invariant. By
  the Borel-Weil theorem the representation $\tau$ is irreducible and
  the map $ \phi : M \ra \PP(V)$ induced by the linear system $|-K_M|$
  is a biholomorphism onto the unique closed orbit in $\PP(V)$.  Let
  $g_{FS}$ be the Fubini-Study metric on $\PP(V)$ induced by $\scalo$,
  and let $\om_{FS}$ be its \Keler form. Since $M$ is a compact
  symmetric space any cohomology class contains a unique invariant
  form \cite[Thm. 8.5.8 p. 250]{wolf-spaces}. Hence $\phi^* \om_{FS} =
  \om_{KE}$. So the data $G, K, \tau, \scalo$ satisfy the same
  assumption set forth at p. \pageref{data} and $(M, \om_{KE})$ is the
  associated flag manifold as defined in \S \ref{section-flags}.  Let
  $\Mom: M \ra \liek^*$ be the moment map of $M$ provided with these
  structures, set $\OO=\Mom(M)$ and let $\intec$ be the interior of
  $\convo$.  Set
  \begin{gather*}
    \mis : = \frac{1}{\vol(N,g)} F\pf \vol_g.
  \end{gather*}
  Any hyperplane $H\subset \PP(V)$ is defined by the vanishing of a
  section $s\in V^*$, so by hypothesis $F\meno (H)$ is a proper
  complex subvariety of $N$.  Therefore $\mis$ is a $\tau$-admissible
  measure on $M$.  By Theorem \ref{grado} $\blym(X) = \intec$ and by
  Lemma \ref{intec-nonvuoto} $0\in \intec$. This means that there some
  $b\in G$ such that
  \begin{gather}
    \label{eq:sottosistema}
    \blym(bK)=\int_M \Mom(\rho(b)\cdot x)d\mis(x) =0.
  \end{gather}
  If we set $a=\rho(b) $ and $h_j = f_j \circ a \circ F$, then
  the above formula yields
  \begin{gather*}
    \int_N h_j (y) \vol_g(y)=0
  \end{gather*}
  for $j=1,\lds, l$. So the functionss $h_j$ have zero mean. By the
  Rayleigh theorem we get
  \begin{gather}
    \label{eq:rayleigh}
    \la_1 (N,g) \leq \frac{\sum_{j=1}^l \int_N |\nabla h_j|^2_g
      \frac{\om^d}{d!} } {\sum_{j=1}^l \int_N h_j^2 \frac{\om^d}{d!}}.
  \end{gather}
  We have to estimate the right hand side.  Since $K$ is transitive,
  the norm $|\Mom|$ is constant. By Lemma \ref{lemmazk} 
  \begin{gather}
    \label{eq:denominatore}
    \sum_{j=1}^l \int_N h_j^2 \frac{\om^d}{d!} = |\Mom|^2 \vol(N,g) =
    \frac{n}{d!} \int_N a^d.
  \end{gather}
  Recall that for any $u\in \cinf(N,\R)$
  \begin{gather*}
    i \partial u \wedge \debar u \wedge \frac{\om^{d-1}} {(d-1)!} =
    \frac{1}{2} |\nabla u |^2 \frac{\om^d}{d!}.
  \end{gather*}
  Hence
  \begin{gather*}
    \sum_{j=1}^l \int_N |\nabla h_j|^2_g \frac{\om^d}{d!}  =
    \frac{2}{(d-1)!} \sum_{j=1}^l \int_N i\partial h_j \wedge \debar
    h_j \wedge \om^{d-1} = \\=\frac{2}{(d-1)!} \int_N F^*( a^*\alfa)
    \wedge \om^{d-1}.
  \end{gather*}
  Since $\alfa$ is closed and $G$ is connected the form $a^*\alfa$ is
  cohomologous to $\alf$. Therefore
  \begin{gather*}
    \sum_{j=1}^l \int_N |\nabla h_j|^2_g \frac{ \om^d}{\ d!}  =
    \frac{4\pi }{ (d-1)!} \int_N F^*( \chern_1(M)) \cup
    a^{d-1}.
  \end{gather*}
  Using \eqref{eq:denominatore} we get the result.
\end{proof}

\begin{remark}
  We wish to compare this result with those contained in
  \cite{arezzo-ghigi-loi} which deal with a similar situation in the
  case of the Grassmannian.  Let $V$ be a vector space of dimension
  $k$. Denote by $M$ the Grassmannian of $r$-planes in $V$, by $U\ra
  M$ the tautological bundle ($U_x=x$) and by $\OO_M(1)$ the
  hyperplane bundle of the Pl\"ucker embedding. Then $V^*=H^0(M, U^*)$,
  $ \Lambda^r U^*= \OO_M(1)$, $-K_M = \OO_M(k)$ and $H^0(M, \OO_M(1))=
  \est^r V^*$.

  Assume that $E \ra N$ is a globally generated holomorphic vector
  bundle of rank $r$ with $H^0(N,E)=V^*$. Consider the map $ F: N \ra
  M$, $ F(x) = \operatorname{Ann}(\{s\in V^*: s(x)=0\})$.  By
  construction $E=F^*(U^*)$. We claim that if the map $F$ satisfies
  the condition of Theorem \ref{teo-sotto}, i.e. no section of $-K_M$
  vanishes on $F(N)$, then the Gieseker point $T_E$ is stable. Indeed,
  since $-K_M=\OO_M(k)$, no section of $\OO_M(1)$ vanishes on $F(N)$,
  hence the map $F^*:H^0(M,\OO_M(1))= H^0(M, \est^r U^*) 
  \ra H^0(N,\est^r E)$ is injective and the Gieseker points of $E$ and
  $U^*$ are related by $T_E=F^*\circ T_{U^*} $. It follows that $T_E$
  is stable.  So in this particular case, where $F$ is the map
  corresponding to the full linear system of sections of $E$, our
  Theorem \ref{teo-sotto} is a consequence of Theorem 1.1 in
  \cite{arezzo-ghigi-loi}.

  More generally, assume that $N$ is a manifold, $M$ is the
  Grassmannian of $r$-planes in some $V$ and $F: N \ra M$ is some map
  satisfying the assumption of Theorem \ref{teo-sotto}.  Since
  $F^*:H^0(M, \OO_M(1))=\est^r V \ra H^0(N,E)$ is injective and
  $\est^r$ commutes with pull-back, also $F^*:H^0(M, U^*)=V \ra
  H^0(N,E)$ is injective.  The constructions in the proof of Theorem
  \ref{teo-sotto} relate to $V=H^0(M, U)$ or, equivalently, to the
  subspace $\im F^* \subset H^0(N,E)$.  Indeed coupled with a theorem
  of Xiaowei Wang (see \cite[Thm. 3.1]{wang-xiaowei-balance},
  \cite[Thm. 2.5]{biliotti-ghigi-AIF}) formula \eqref{eq:sottosistema}
  says roughly that $\im F^*$ is ''$\om$-balanced'' (see
  \cite[p. 380]{arezzo-ghigi-loi}).  This does not imply that the
  whole space $H^0(N,E)$ be $\om$-balanced (which is equivalent to
  $T_E$ being stable).  So for general $F$ the hypothesis of Theorem
  \ref{teo-sotto} is weaker than the assumption of Theorem 1.1 in
  \cite{arezzo-ghigi-loi}.
\end{remark}

\begin{cor}
  Let $M$ be a Hermitian symmetric space of compact type.  If $g$ is a
  \Keler metric with \Keler form $\om \in 2\pi \chern_1(M)$, then
  \begin{gather*}
    \la_1(M,g) \leq 2.
  \end{gather*}
  The bound is attained by the symmetric metric.
\end{cor}
\begin{proof}
  To get the estimate it is enough to apply the theorem with $N=M$,
  $F=\Id$ and $a=2\pi\chern_1(M)$.  The bound is attained by $g_{KE}$
  by Theorem \ref{futaki}.
\end{proof}

\def\cprime{$'$}

\end{document}